\documentclass[12pt]{article}

\usepackage[latin1]{inputenc}
\usepackage{subfigure}
\usepackage{amsmath,amssymb,amsthm}
\usepackage[mathscr]{eucal}
\usepackage{eepic}
\usepackage{gastex}
\usepackage{cite}
\usepackage{graphicx}

\newtheorem{Lemma}{Lemma}
\newtheorem{Proposition}{Proposition}
\newtheorem{Theorem}{Theorem}
\newtheorem{Example}{Example}

\newtheorem{Corollary}{Corollary}

\newcommand{\is}{\approx}
\newcommand{\ov}{\overline}
\newcommand{\wh}{\widehat}
\DeclareMathOperator{\integers}{\mathbb{N}}
\DeclareMathOperator{\Sing}{\mathrm{Sing}}

\author{Inna Mikhailova\footnote{Ural Federal University, Ekaterinburg, Russia.}}
\title{ A proof of Zhil'tsov's theorem on decidability of equational theory of epigroups}

\begin{document}
\maketitle
\begin{abstract}
Epigroups are semigroups equipped with an additional unary
operation called pseudoinversion. Each finite semigroup can be
considered as an epigroup. We prove the following theorem announced
by Zhil'tsov in 2000: the equational theory of the class of all
epigroups coincides with the equational theory of the class of all
finite epigroups and is decidable. We show that the theory is not
finitely based but provide a transparent infinite basis for it.
\end{abstract}

\section{Introduction}
\label{sec:in} A semigroup $S$ is called an \emph{epigroup} if,
for every element $x\in S$, some power of $x$ belongs to a
subgroup of $S$. The class of all epigroups includes
all periodic semigroups (i.e., semigroups in which each element
has an idempotent power), all completely regular semigroups (i.e.,
unions of groups), and many other important classes of semigroups.
See \cite{Shevrin94_1,Shevrin94_2} and the
survey~\cite{Shevrin_survey} for more examples and an introduction
to the structure theory of epigroups.

It is known that, for every element $x$ of an epigroup $S$, there
exists a unique maximal subgroup $G_x$ that contains all but
finitely many powers of $x$. Let $e_x$ stand for the identity
element of $G_x$. Then it is known that $xe_x=e_xx$ and that the
product belongs to $G_x$. The latter fact allows one to consider
the inverse of $xe_x$ in the group $G_x$; we denote this inverse
by $\ov{x}$. This defines the unary operation $x\mapsto\ov{x}$ on
each epigroup; we call this operation \emph{pseudoinversion}.
Thus, epigroups can be treated as unary semigroups, that is, as
algebras with two operations: multiplication and pseudoinversion,
and we shall adopt this meaning of the term `epigroup' throughout.
Let $\mathfrak{E}$ stand for the class of all epigroups.

A systematic study of epigroups as unary semigroups was initiated
by Lev Shevrin in \cite{Shevrin94_1,Shevrin94_2,Shevrin_survey}.
In particular, Shevrin suggested to investigate the
collection of all unary semigroup identities holding in all
epigroups, that is, the \emph{equational theory} of $\mathfrak{E}$
as a class of unary semigroups\footnote{In
order to prevent any chance of confusion, we mention here that the
class $\mathfrak{E}$ does not form a variety of unary semigroups;
of course, this is not an obstacle for considering the equational
theory of $\mathfrak{E}$.}. The most fundamental questions about
this theory are whether or not it is decidable and whether or not
it is finitely axiomatizable.

Yet another motivation for studying unary identities of epigroups
comes from the theory of finite semigroups. Every finite semigroup
can be treated as an epigroup, and the unary operation of
pseudoinversion is an \emph{implicit operation} in the sense of
Jan Reiterman~\cite{Reiterman}, that is, it commutes with the homomorphisms
between finite semigroups. Therefore, for each collection $\Sigma$ of
unary identities of epigroups, the class of all finite semigroups
satisfying $\Sigma$ is a pseudovariety. In fact, many important
pseudovarieties can be defined by identities involving the operation
of pseudoinversion (which within the realm of finite semigroups is
usually denoted by $x\mapsto x^{\omega-1}$) or the operation
$x\mapsto x^\omega:=\ov{x}x$; see \cite{Almeida}
for plentiful examples.

We denote by $\mathfrak{E}_{fin}$ the class of all finite epigroups.  Along with $\mathfrak{E}$ and $\mathfrak{E}_{fin}$, among important classes of epigroups is the class $\mathfrak{A}_{fin}$ consisting of finite combinatorial (or aperiodic) semigroups. Recall that a semigroup is \emph{combinatorial} if all its maximal subgroups are one-element.

At the end of the 1990s Ilya Zhil'tsov began to study the decidability problem for the equational theory of $\mathfrak{A}_{fin}$ (considered as a class of epigroups) along with a more general question. First results he obtained were published in~\cite{ZH_art} and a full solution of the problem for $\mathfrak{A}_{fin}$ (found independently of Jon McCammond's solution in~\cite{McM}) was announced in~\cite{ZH_docl}. The paper~\cite{ZH_docl} also contained similar results for the pseudovariety $\mathfrak{E}_{fin}$. We collect these results in the following statement:

\begin{Theorem}\label{THM_Main}
The equational theory of the class $\mathfrak{E}$ coincides with the equational theory of the class $\mathfrak{E}_{fin}$ and is decidable. The theory is not finitely based but has the following infinite identity basis:
\begin{gather*}
   (xy)z\is x(yz),\\
   \ov{(xy)}x\is x\ov{(yx)},\\
    \ov{x}^2x\is \ov{x},\\
    x^2\ov{x}\is\ov{\ov{x}},\\
    \ov{\ov{x}x}\is\ov{x}x,\\
    \ov{x^p}\is \ov{x}^p \hbox to 0mm{\quad\text{for each prime $p$.}}
\end{gather*}
\end{Theorem}

Very unfortunately, soon after the announcement~\cite{ZH_docl} had appeared, Zhil'tsov died in a tragic accident and left no implementation of the statements indicated in~\cite{ZH_docl}. It took us considerable effort to reconstruct all necessary steps of the proof. In Sections~\ref{sec:Z-unary words} and~\ref{sec:Normal swords} we follow Zhil'tsov's plan outlined in~\cite{ZH_docl} quite closely while in Section~\ref{sec:Proof of the theorem} we choose a somewhat different way.

In Section~\ref{sec:Z-unary words} we introduce Zhil'tsov's concept of a $\mathbb{Z}$-unary word. He suggested to consider words with not just one but countably many additional unary operations.  In Section~\ref{sec:Normal swords}, we consider normal forms for $\mathbb{Z}$-unary words and prove two of Zhil'tsov's propositions. First, it is possible to algorithmically construct the normal form. Second, there is an algorithm that, for two  $Z$-words in normal form, returns their longest common prefix [suffix].  In particular, this means that there is an algorithm that decides whether the normal forms of given $\mathbb{Z}$-unary words coincide. In Section~\ref{sec:Proof of the theorem} we consider epigroup terms as $\mathbb{Z}$-unary words and show that an epigroup identity holds in each epigroup if and only if the normal form of the left-hand side of the identity coincides with the normal form of its right-hand side.

When the results forming this paper had already been obtained and the paper was being prepared for publication, the author learned about Jos\'e Carlos Costa's preprint \cite{Costa_preprint} also dealing with the equational theory of $\mathfrak{E}_{fin}$; later Costa's paper~\cite{Costa} was published. Let us briefly comment on the relation between~\cite{Costa} and the present paper. Both papers provide algorithms to decide the equational theory of $\mathfrak{E}_{fin}$ and bases for the theory. However, the two algorithms utilize essentially different approaches and their justifications come from different sources. The bases found in~\cite{Costa} and in the present paper are, of course, equivalent.

\section{$\mathbb{Z}$-unary words}
\label{sec:Z-unary words}

\subsection{Basic definitions}

We fix an \emph{alphabet} $A$, that is, a non-empty set which elements are referred to as \emph{letters}. As usual, by $A^+$ we denote the \emph{free semigroup} over $A$, that is, the set of all non-empty \emph{words} over $A$ which are multiplied by concatenation. We extend $A^+$ to a larger algebra that we denote by $Z(A)$ and call the \emph{free} $Z$-\emph{unary semigroup} over $A$, the elements of $Z(A)$ being called $Z$-\emph{unary words} over $A$. For this we fix a symbol $\omega$ and define the notions of a $\mathbb{Z}$-unary word and of its \emph{height} by simultaneous induction as follows:
\begin{enumerate}
 \item[1)] the empty word is a $\mathbb{Z}$-unary word of height $-\infty$;
 \item[2)] every letter from $A$ is a $\mathbb{Z}$-unary word of height $0$;
 \item[3)] for every integer $q\in\mathbb{Z}$ and every $\mathbb{Z}$-unary word $\sigma$ of height $h$,
 the expression $(\sigma)^{\omega+q}$ is a $\mathbb{Z}$-unary word of height $h+1$;
 \item[4)] for every pair $\sigma_1,\sigma_2$ of $\mathbb{Z}$-unary words of heights $h_1$ and respectively $h_2$, the expression $\sigma_1\sigma_2$ is a $\mathbb{Z}$-unary word of height $\max\{h_1,h_2\}$.
\end{enumerate}
For example, the expression
\[
(x^{\omega-4}yx^{\omega+30})^{\omega-1}xy^{\omega}
\]
is a $\mathbb{Z}$-unary word over $\{x,y\}$ of height 2. This example also illustrates three natural conventions that we adopt throughout: we omit parentheses in expressions like $(\sigma)^{\omega+q}$ whenever $\sigma$ is just a letter and take the liberty to write $\omega$ instead of $\omega+0$ and $\omega-q$ instead of $\omega+(-q)$ for $q$ being a positive integer.

We denote the height of $\sigma\in Z(A)$ by $h(\sigma)$. It is easy to verify that every $\mathbb{Z}$-unary word $\sigma$ of height $h+1$ with $h\ge 0$ can be uniquely represented as
\begin{equation}
\label{eq:represent}
\sigma=\pi_0\rho_1^{\omega+q_1}\pi_1\cdots \rho_n^{\omega+q_n}\pi_n,
\end{equation}
where $n\geq 1$, $h(\rho_i)=h$ for all $i=1,\dots,n$, and $h(\pi_i)\le h$ for all $i=0,\dots,n$.
We call \eqref{eq:represent} the \emph{height representation} of $\sigma$.

We define the mapping $\sigma\mapsto|\sigma|$ from the free $\mathbb{Z}$-unary semigroup $Z(A)$ into the ring
$\mathbb{Z}[\omega]$ of all polynomials in $\omega$ with integer coefficients as follows:
\begin{enumerate}
  \item[1)] $|\sigma|:=0$ if $\sigma$ is the empty word;
  \item[2)] $|\sigma|:=1$ whenever $\sigma$ is a letter from $A$;
  \item[3)] $|\sigma^{\omega+q}|:=(\omega+q)|\sigma|$ for every integer $q\in\mathbb{Z}$ and every $\sigma\in Z(A)$;
  \item[4)] $|\sigma\tau|:=|\sigma|+|\tau|$ for all $\sigma,\tau\in Z(A)$.
\end{enumerate}
We call $|\sigma|$ the \emph{length} of the $\mathbb{Z}$-unary word $\sigma$. Observe that $h(\sigma)$ is just is the degree $\deg|\sigma|$ of the polynomial $|\sigma|$. For two polynomials $f,g\in\mathbb{Z}[\omega]$, we write $f\ge g$ if the leading coefficient of the polynomial $f-g$ is non-negative. Then $|\sigma|\ge0$ for every $\sigma\in Z(A)$.

We say that a $\mathbb{Z}$-unary word $\tau$ is a \emph{prefix} [resp. \emph{suffix}] of $\sigma\in Z(A)$ if $\sigma=\tau\rho$ [resp. $\sigma=\rho\tau$] for some $\rho\in Z(A)$. A $\mathbb{Z}$-unary word $\tau$ is a \emph{factor} of $\sigma$ if $\sigma=\rho_1\tau\rho_2$ for some $\rho_1,\rho_2\in Z(A)$ and it is a \emph{power} of $\sigma$ if $\tau=\sigma^n$ for some $n\ge2$. Likewise, $\mathbb{Z}$-unary words of the form $\sigma^{\omega+q}$, where $q$ is an integer, are called $\omega$-\emph{powers} of $\sigma$. Observe that the free $\mathbb{Z}$-unary semigroup $Z(A)$ considered as a semigroup is just the free monoid generated by $A$ and all $\omega$-powers.

\subsection{Singular words}

Consider the fully invariant congruence $\mathcal{S}$ on the free $\mathbb{Z}$-unary semigroup $Z(A)$ generated by the pairs
  \begin{gather}
    (x^{\omega+q}, xx^{\omega+q-1}),\label{eq:wind1}\\
    (x^{\omega+q}, x^{\omega+q-1}x),\label{eq:wind2}\\
    (x(yx)^{\omega+q}, (xy)^{\omega+q}x)\label{eq:roll},
  \end{gather}
where $q$ runs over $\mathbb{N}$.

By a \emph{singular word} or, shortly, a \emph{sword} we mean an $\mathcal{S}$-class. The class corresponding to a given $\mathbb{Z}$-unary word $\sigma$ is denoted by $\sigma^{\mathcal{S}}$. Notice that if $\sigma$ is an ordinary word, that is, $\sigma\in A^+$, then $\sigma^{\mathcal{S}}=\{\sigma\}$. The quotient algebra $Z(A)/\mathcal{S}$ of all swords is denoted by $\Sing(A)$.

\begin{Example}[Zhil'tsov]
Consider the $\mathbb{Z}$-unary word $x(xyz)^{\omega-5}xy(zx)^{\omega+4}zz$ of height \textup1 and of length $5\omega-2$. We can represent it by the following picture:
\begin{center}
\includegraphics{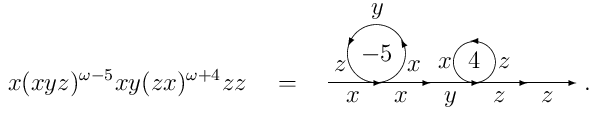}\\
\emph{Figure 1: Pictorial representation if a $\mathbb{Z}$-unary word}\\[1ex]
\end{center}
Here the $\omega$-powers $(xyz)^{\omega-5}$ and $(zx)^{\omega+4}$ are represented by the circles labeled $xyz$ and $zx$ with encircled numbers $-5$ and $4$ respectively. Now we draw the analogous pictures for three further $\mathbb{Z}$-unary words that belong to the same $\mathcal{S}$-class.
\begin{center}
\includegraphics{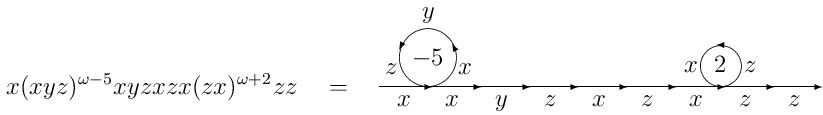}\\
\emph{Figure 2: Unwinding}\\[1ex]
\includegraphics{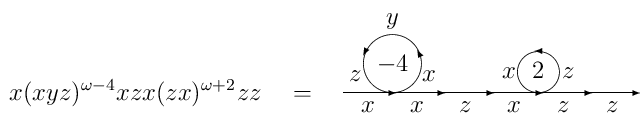}\\
\emph{Figure 3: Winding up}\\[1ex]
\includegraphics{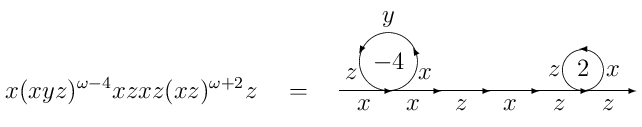}\\
\emph{Figure 4: Rolling}\\[1ex]
\end{center}
These pictures illustrate three basic transformations that one can perform on a $\mathbb{Z}$-unary word without changing its  $\mathcal{S}$-class.
    \begin{itemize}
      \item One can ``unwind'' $n$ copies of a circle to the left or the right side of the circle, simultaneously decreasing its encircled number by $n$ --- in Fig.~2 this transformation is applied to the circle labeled $zx$ in Fig.~1 with $n=2$. This corresponds to $n$ applications of either \eqref{eq:wind1} or  \eqref{eq:wind2} from left to right.
      \item One can ``wind up'' (if it is possible) $n$ copies of a circle from the left or the right side of the circle, simultaneously increasing its encircled number by $n$  --- in Fig.~3 this transformation is applied to the circle labeled $xyz$ in Fig.~2 with $n=1$. This corresponds to $n$ applications of either \eqref{eq:wind1} or  \eqref{eq:wind2} from right to left.
      \item One can ``roll'' a circle to the left (right) side if some suffix (resp. prefix) of the circle's label is written before (resp. after) the circle --- in Fig.~4 this transformation is applied to the circle labeled $zx$ in Fig.~3. This corresponds to an application of \eqref{eq:roll}.
    \end{itemize}
\end{Example}

Thus, any sword can be considered as an ordinary word with some attached ``singularities'' ($\omega$-powers that may contain further $\omega$-powers etc.) that can move along the word. This explains our terminology.

We refer to the basic transformations illustrated in Example~1 as $\mathcal{S}$-\emph{transformations}. Each $\mathcal{S}$-trans\-for\-ma\-tion involves a certain $\omega$-power which we call the \emph{site} of the transformation. Clearly, two $\mathbb{Z}$-unary words belong to the same  $\mathcal{S}$-class if and only if one can obtain each of these words form the other by a suitable sequence of $\mathcal{S}$-transformations.

Notice that, due to the definition of the congruence $\mathcal{S}$, the following conditions hold
for all $\mathbb{Z}$-unary words $\sigma,\tau$ in the same $\mathcal{S}$-class:
  \begin{itemize}
   \item $\sigma$ and $\tau$ have the same height;
   \item $\sigma$ and $\tau$ have the same length.
  \end{itemize}
Thus, we can well define the \emph{height} and the \emph{length} of a sword $\sigma^{\mathcal{S}}$ as the height and respectively the length of $\sigma$. We can also speak about a height representation of a sword; observe, however, that in contrast to Z-unary words, a sword may have several height representations.

We say that a sword $\tau^{\mathcal{S}}$ is a \emph{prefix} [\emph{suffix}, \emph{factor}] of a sword $\sigma^{\mathcal{S}}$ if, for some $\mathbb{Z}$-unary words $\tau',\sigma'$ such that $(\tau')^\mathcal{S}=\tau^{\mathcal{S}}$ and $(\sigma')^\mathcal{S}=\sigma^{\mathcal{S}}$, the $\mathbb{Z}$-unary word $\tau'$ is a prefix [resp.\ suffix, factor] of $\sigma'$. In the same way, the notions of a power and of an $\omega$-power extend to swords.

As we will see, the algebra $\Sing(A)$ considered as a semigroup inherits several important properties from the free semigroup. Of course, there are some differences too. For instance,  each sword of height more than zero has infinitely many factors. Further, for every ordinary word $w$ and every integer $t$ such that $0\le t\le |w|$, there exists a prefix of $w$ with length $t$. This is not true in general swords: for example, the sword $((x^2)^{\omega})^{\mathcal{S}}$ has no prefix of length $\omega$.

\begin{Lemma}\label{Lemma_Find_Prefix_of_certain_length}
There is an algorithm that, given an arbitrary sword $\sigma^{\mathcal{S}}$ and a polynomial $f(\omega)\in\mathbb{Z}[w]$ such that the leading coefficient of $|\sigma|-f(\omega)$ is non-negative, decides whether $\sigma^{\mathcal{S}}$ has a prefix of length $f(\omega)$. If the answer is positive, the algorithm constructs swords $x^{\mathcal{S}},y^{\mathcal{S}}$ such that $\sigma^{\mathcal{S}}=x^{\mathcal{S}}y^{\mathcal{S}}$ and $|x|=f(\omega)$.
\end{Lemma}

\begin{proof}
We induct on the height of $\sigma$. As we have noticed prior to the formulation of the lemma, its statement holds for ordinary words and hence for swords of height $0$.

Let $h(\sigma)=h+1$ and let \eqref{eq:represent} be the height representation of $\sigma$. Since $f(\omega)\leq |\sigma|$, there is an index $k$ such that
\[
|\pi_0\cdots\rho_k^{\omega+q_k}|\leq f(\omega)\leq |\pi_0\cdots\rho_k^{\omega+q_k}\pi_{k}\rho_{k+1}^{\omega+q_{k+1}}|.
\]
(If $\deg f(\omega)<h+1$, then $k=0$ and the above inequality takes the form $0\le f(\omega)\le|\pi_0\rho_1^{\omega+q_1}|$.) Let $g(\omega)=f(\omega)-|\pi_0\cdots\rho_k^{\omega+q_k}|$. (In the case where $k=0$, we let $g(\omega)=f(\omega)$.) Clearly, $\sigma^{\mathcal{S}}$ has a prefix of length $f(\omega)$ if and only if the sword $(\pi_k\rho_{k+1}^{\omega+q_{k+1}})^{\mathcal{S}}$ has a prefix of length $g(\omega)$.

First suppose that $\deg g(\omega)< h+1$. Then it is possible to choose a positive integer $m$ such that
the leading coefficient of $|\pi^k\rho_{k+1}^m|-g(\omega)$ is non-negative, and we are in a position to apply the induction hypothesis to the sword $(\pi^k\rho_{k+1}^m)^{\mathcal{S}}$ and the polynomial $g(\omega)$. If the answer is positive and $(\pi^k\rho_{k+1}^m)^{\mathcal{S}}=x_1^{\mathcal{S}}y_1^{\mathcal{S}}$ with $|x_1|=g(\omega)$, then the required swords are $x^{\mathcal{S}}=(\pi_0\cdots\rho_k^{\omega+q_k}x_1)^{\mathcal{S}}$ and $y^{\mathcal{S}}=(y_1\rho_{k+1}^{\omega+q_{k+1}-m}\cdots \pi_n)^{\mathcal{S}}$.

Now suppose that $\deg g(\omega)=h+1$. It is easy to see that if the required prefix exists then it has the form $(\pi_k\rho_{k+1}^{\omega+t}\eta)^{\mathcal{S}}$ for some integer $t\le q_{k+1}$ and some prefix $\eta$ of the $\mathbb{Z}$-unary word $\rho_{k+1}$. Since $h(\pi_k)\le h$ and $h(\eta)\le h$, we conclude that, by the definition of the length, the leading coefficients of the polynomials $|\pi_k\rho_{k+1}^{\omega+t}\eta|$, $|\pi_k\rho_{k+1}^{\omega+q_{k+1}}|$, and $g(\omega)$ must coincide. If this holds, then we choose an integer $s<q_{k+1}$ such that the leading coefficient of the polynomial $g'(\omega)=|\pi_k\rho_{k+1}^{\omega+s}|-g(\omega)$ is non-negative. Since $\deg g'(\omega)<h+1$, we can apply the induction hypothesis to the sword $(\rho_{k+1}^{q_{k+1}-s})^{\mathcal{S}}$ and the polynomial $g'(\omega)$, obtaining either a negative answer or a decomposition $(\rho_{k+1}^{q_{k+1}-s})^{\mathcal{S}}=x_1^{\mathcal{S}}y_1^{\mathcal{S}}$. Then the required prefix is $x^{\mathcal{S}}=(\pi_0\cdots\rho_k^{\omega+q_k}\pi_k\rho_{k+1}^{\omega+s}x_1)^{\mathcal{S}}$.
\end{proof}

\begin{Lemma}\label{LemmaSingProperties}
\emph{1.} $\Sing(A)$ is a $\mathscr{J}$-trivial cancellative semigroup;

\emph{2.} The set of all prefixes \textup[suffixes\textup] of a sword is linearly ordered by right \textup[resp.\ left\textup] division.
\end{Lemma}

\begin{proof}
Let us start with the proof of the cancelation property. Suppose that $\sigma^{\mathcal{S}}\tau_1^{\mathcal{S}}
=\sigma^{\mathcal{S}}\tau_2^{\mathcal{S}}$. This implies $(\sigma\tau_1)^{\mathcal{S}}=(\sigma\tau_2)^{\mathcal{S}}$.
Therefore there exists a sequence of $\mathbb{Z}$-unary words $\sigma\tau_1=\eta_1,\eta_2,\dots,\eta_n=\sigma\tau_2$ such that for each $i=1,\dots,n-1$, the word $\eta_{i+1}$ can be obtained from the word $\eta_i$ by one of the $\mathcal{S}$-transformations: ``unwinding'', ``winding up'', or ``rolling''. We will construct two sequences of $\mathbb{Z}$-unary words $\alpha_1,\dots,\alpha_n$ and $\beta_1,\dots,\beta_n$ with the following properties:
\begin{itemize}
\item[(i)] $\alpha_1=\sigma$, $\beta_1=\tau_1$;
\item[(ii)] for each $i=1,\dots,n-1$, one can pass from $\alpha_i$ to $\alpha_{i+1}$ and from $\beta_i$ to $\beta_{i+1}$ by a sequence of $\mathcal{S}$-transformations;
\item[(iii)] for each $i=1,\dots,n$, the $\mathbb{Z}$-unary word $\alpha_i\beta_i$ can be obtained from $\eta_i$ by applying only ``unwinding'' $\mathcal{S}$-transformations.
\end{itemize}
This will imply the desired conclusion $\tau_1^{\mathcal{S}}=\tau_2^{\mathcal{S}}$. Indeed, by (i) and (ii), we have $\sigma=\alpha_1\mathrel{\mathcal{S}}\alpha_n$ and $\tau_1=\beta_1\mathrel{\mathcal{S}}\beta_n$ whence, in particular, $|\sigma|=|\alpha_n|$. By (iii), $\alpha_n\beta_n$ can be obtained from $\eta_n=\sigma\tau_2$ by ``unwinding'' $\mathcal{S}$-transformations only. The definition of an ``unwinding'' $\mathcal{S}$-transformation implies that if such a transformation is applied to a product of two $\mathbb{Z}$-unary words,
its site is contained in one of the factors. Thus, we conclude that $\alpha_n\beta_n=\gamma\delta$, where $\gamma$ and $\delta$ are obtained by some ``unwinding'' $\mathcal{S}$-transformations from $\sigma$ and $\tau_2$ respectively. In particular, $\sigma\mathrel{\mathcal{S}}\gamma$ and $\tau_2\mathrel{\mathcal{S}}\delta$. Since $|\gamma|=|\sigma|=|\alpha_n|$, we have $\gamma=\alpha_n$ whence $\delta=\beta_n$. This yields
$\tau_1\mathrel{\mathcal{S}}\beta_n=\delta\mathrel{\mathcal{S}}\tau_2$ and $\tau_1\mathrel{\mathcal{S}}\tau_2$, as required.

Thus, it remains to construct the sequences $\alpha_1,\dots,\alpha_n$ and $\beta_1,\dots,\beta_n$ satisfying (i)--(iii). We proceed by induction, using (i) as the induction basis. Suppose that the $\mathbb{Z}$-unary words $\alpha_i$ and $\beta_i$ have already been defined. If the $\mathcal{S}$-transformation $\Phi$ applied to obtain $\eta_{i+1}$ from $\eta_i$ is of the type ``winding up'', we let $\alpha_{i+1}=\alpha_i$, $\beta_{i+1}=\beta_i$. If $\Phi$ is of the type ``unwinding'', one of the two possibilities occur: either $\Phi$ is one of the $\mathcal{S}$-transformations employed to convert $\eta_i$ into $\alpha_i\beta_i$ or not. In the former case
we again let $\alpha_{i+1}=\alpha_i$, $\beta_{i+1}=\beta_i$. In the latter case, the site of $\Phi$ persists in $\alpha_i\beta_i$ and, as already mentioned, it is located within one of the factors $\alpha_i$ or $\beta_i$.
If the site lies within $\alpha_i$, we let $\alpha_{i+1}$ be the result of applying $\Phi$ to $\alpha_i$ and let $\beta_{i+1}=\beta_i$; otherwise we let $\alpha_{i+1}=\alpha_i$ and let $\beta_{i+1}$ be the result of applying $\Phi$ to $\beta_i$.

Now let $\Phi$ be of the type ``rolling''. By symmetry, we may assume that $\Phi$ corresponds to an application of \eqref{eq:roll} from right to left, that is, the site $(xy)^{\omega+q}$ of $\Phi$ is followed by $x$ in $\eta_i$ and ``rolls'' to the right producing $x(yx)^{\omega+q}$ in $\eta_{i+1}$. Since $\alpha_i\beta_i$ is obtained from $\eta_i$ by ``unwinding'' $\mathcal{S}$-transformations, $\alpha_i\beta_i$ may contain several copies of the site of $\Phi$---for illustration, if, say, $\eta_i=(z(xy)^{\omega+q}x)^{\omega+r}$, then unwinding the external $\omega$-power gives $\mathbb{Z}$-unary words of the form $(z(xy)^{\omega+q}x)^s(z(xy)^{\omega+q}x)^{\omega+r-s-t}(z(xy)^{\omega+q}x)^t$, with $s+t+1$ occurrences of the site. We apply $\Phi$ to all copies of $(xy)^{\omega+q}x$ that occur within $\alpha_i$ or $\beta_i$. Besides that,
it may happen that some copies of the site have undergone unwinding themselves, resulting in a $\mathbb{Z}$-unary word of the form $\xi=(xy)^r(xy)^{\omega+q-r-s}(xy)^sx$. Whenever such a $\xi$ lies entirely within $\alpha_i$ or $\beta_i$, we substitute it by $x(yx)^r(yx)^{\omega+q-r-s}(yx)^s$---clearly, this amounts to ``rolling'' of $(xy)^{\omega+q-r-s}$ to the right. If $\xi$ occurs on the junction of $\alpha_i$ and $\beta_i$, we can basically do the same except for the only case where $(xy)^r(xy)^{\omega+q-r-s}$ is a suffix of $\alpha_i$ while $(xy)^sx$ is a prefix of $\beta_i$ because in this case we cannot ``roll'' $(xy)^{\omega+q-r-s}$ to the right within $\alpha_i$. In this remaining case, when passing from $\alpha_i$ to $\alpha_{i+1}$, we substitute $(xy)^r(xy)^{\omega+q-r-s}$ by $x(yx)^r(yx)^{\omega+q-r-s-1}y$. This concludes the induction step.

Now let us show that $\Sing(A)$ is $\mathscr{J}$-trivial. Indeed, assume that some swords $\alpha^{\mathcal{S}},\beta^{\mathcal{S}}$ are $\mathscr{J}$-related. Then $\alpha^{\mathcal{S}}=\sigma_1^{\mathcal{S}}\beta^{\mathcal{S}}\tau_1^{\mathcal{S}}$ and $\beta^{\mathcal{S}}=\sigma_2^{\mathcal{S}}\alpha^{\mathcal{S}}\tau_2^{\mathcal{S}}$ for some
$\sigma_1,\sigma_2,\tau_1,\tau_2\in Z(A)$. We have $\alpha^{\mathcal{S}}=
\sigma_1^{\mathcal{S}}\sigma_2^{\mathcal{S}}\alpha^{\mathcal{S}}\tau_2^{\mathcal{S}}\tau_1^{\mathcal{S}}$ whence $|\sigma_1|=|\sigma_2|=|\tau_1|=|\tau_2|=0$ and all these $\mathbb{Z}$-unary words must be empty. Hence $\alpha^{\mathcal{S}}=\beta^{\mathcal{S}}$.

The $\mathscr{J}$-triviality of the semigroup $\Sing(A)$ implies that the relations of right and left division are partial orders on $\Sing(A)$. We aim to show that their restrictions to the set of all prefixes [suffixes] of a fixed sword $\sigma^{\mathcal{S}}$ are linear orders. By symmetry, it suffices to consider right division only. Thus, suppose that $\sigma^{\mathcal{S}}=\tau_1^{\mathcal{S}}\rho_1^{\mathcal{S}}=\tau_2^{\mathcal{S}}\rho_2^{\mathcal{S}}$ and $|\tau_1|\le|\tau_2|$.  Since $\tau_1\rho_1\mathrel{\mathcal{S}}\tau_2\rho_2$, the sword $(\tau_2\rho_2)^{\mathcal{S}}$ has a prefix of length $|\tau_1|$ that should be also a prefix of the sword $(\tau_2)^{\mathcal{S}}$. Hence, using the algorithm of Lemma~\ref{Lemma_Find_Prefix_of_certain_length}, we can find two swords $x^{\mathcal{S}}$ and $y^{\mathcal{S}}$ such that  $\tau_2^{\mathcal{S}}=x^{\mathcal{S}}y^{\mathcal{S}}$ and $|x|=|\tau_1|$.

Since $\tau_1\rho_1\mathrel{\mathcal{S}}xy\rho_2$, there exists a sequence of $\mathbb{Z}$-unary words $\eta_1=\tau_1\rho_1, \eta_2,\dots,\eta_n=xy\rho_2$ such that $\eta_{i+1}$ can be obtained from $\eta_i$ by an application of a suitable $\mathcal{S}$-transformation for each $i=1,\dots,n-1$. Now we follow the proof of the cancelation property above and construct two sequences of $\mathbb{Z}$-unary words $\alpha_1,\dots,\alpha_n$ and $\beta_1,\dots,\beta_n$ such that
\begin{itemize}
\item[(i)] $\alpha_1=\tau_1$, $\beta_1=\rho_1$;
\item[(ii)] for each $i=1,\dots,n-1$, one can pass from $\alpha_i$ to $\alpha_{i+1}$ and from $\beta_i$ to $\beta_{i+1}$ by a sequence of $\mathcal{S}$-transformations;
\item[(iii)] for each $i=1,\dots,n$, the $\mathbb{Z}$-unary word $\alpha_i\beta_i$ can be obtained from $\eta_i$ by applying only ``unwinding'' $\mathcal{S}$-transformations.
\end{itemize}
Then the same argument as the one we have utilized in the proof of the cancelation property shows that $\tau_1=\alpha_1\mathrel{\mathcal{S}}\alpha_n\mathrel{\mathcal{S}}x$ whence $\tau_1^{\mathcal{S}}=x^{\mathcal{S}}$ divides $\tau_2^{\mathcal{S}}$ on the right.
\end{proof}

\section{Normal swords}
\label{sec:Normal swords}

\subsection{Overview}

In this section we introduce a rewriting system that plays a key role for the decidability question.

Let $\mathcal{R}$ denote the rewriting system on the free $\mathbb{Z}$-unary semigroup $Z(A)$ consisting of the rules
 \begin{gather}
   x^{\omega+q}x^{\omega+r}\stackrel{\mathcal{R}}{\longrightarrow} x^{\omega+q+r},\label{eq:addition}\\
   (x^n)^{\omega+q}\stackrel{\mathcal{R}}{\longrightarrow} x^{\omega+nq},\label{eq:multiplication1}\\
   (x^{\omega+r})^{\omega+q}\stackrel{\mathcal{R}}{\longrightarrow} x^{\omega+rq}\label{eq:multiplication2}
 \end{gather}
for all $x\in Z(A)$, all $r,q\in\mathbb{N}$, and all $n\geq 2$. Observe that all these rules are length-decreasing; the rule~\eqref{eq:multiplication2} is also height-decreasing while~\eqref{eq:addition} and~\eqref{eq:multiplication1} do not change the height. The rewriting system $\mathcal{R}$ induces the ``quotient'' rewriting system $\mathcal{R}/\mathcal{S}$ on the quotient semigroup $\Sing(A)=Z(A)/\mathcal{S}$ as follows:
$\alpha^{\mathcal{S}}\stackrel{\mathcal{R}/\mathcal{S}}{\longrightarrow} \beta^{\mathcal{S}}$ if and only if there exist $\mathbb{Z}$-unary words $\sigma$ and $\tau$ such that $\alpha^{\mathcal{S}}=\sigma^{\mathcal{S}}$,
$\beta^{\mathcal{S}}=\tau^{\mathcal{S}}$, and $\sigma\stackrel{\mathcal{R}}{\longrightarrow}\tau$. By $(\mathcal{R}/\mathcal{S})^*$ we denote the reflexive and transitive closure of $\mathcal{R}/\mathcal{S}$.

A \emph{normal sword} (or an $\mathcal{R}/\mathcal{S}$-\emph{irreducible sword}) is a sword $\alpha^{\mathcal{S}}$ such that there is no sword $\beta^{\mathcal{S}}$ with $\alpha^{\mathcal{S}}\stackrel{\mathcal{R}/\mathcal{S}}{\longrightarrow} \beta^{\mathcal{S}}$. Observe that every factor of a normal sword is also normal. A sword $\alpha^{\mathcal{S}}$ is called \emph{fully normal} if its square $(\alpha^{\mathcal{S}})^2$ is normal. It is easy to see that then $(\alpha^{\mathcal{S}})^n$ is normal for all $n\geq 2$.

Any normal sword which is $(\mathcal{R}/\mathcal{S})^*$-related to a given sword $\sigma^{\mathcal{S}}$ is said to be a \emph{normal form} of $\sigma^{\mathcal{S}}$. We stress that at this point we claim no uniqueness nor confluence: a sword may have several normal forms.

Up to now, we have always tried to differentiate between a sword (i.e., an $\mathcal{S}$-class) and a $\mathbb{Z}$-unary word representing this $\mathcal{S}$-class. In the sequel, in order to lighten the notation, we allow ourselves to omit the superscript ${}^\mathcal{S}$ so that a sword, say, $\alpha^{\mathcal{S}}$ may be denoted by just $\alpha$.

This definition of a normal sword implies that such a sword contains no factors of the forms $x^{\omega+r}x^{\omega+q}$, $(x^n)^{\omega+q}$, $(x^{\omega+r})^{\omega+q}$. Notice also that if a sword is fully normal, then it is obviously normal and cannot have the form $x^{\omega+m}$ nor $x^{\omega+r}yx^{\omega+q}$. Clearly, all ordinary words are fully normal.

The main goal of this section is to prove the two following propositions.
\begin{Proposition}\label{Prop_Normalization}
There exists an algorithm that, given a sword $\sigma$ of height $h$, returns a normal form of $\sigma$.
\end{Proposition}
\begin{Proposition}\label{Prop_LongestCommonPrefix}
There exists an algorithm that, given two normal swords of height at most $h$, returns their longest common prefix.
\end{Proposition}

We have included the parameter $h$ in the formulations of Propositions~\ref{Prop_Normalization} and~\ref{Prop_LongestCommonPrefix} because we are going to prove these propositions, along with Lemmas~\ref{Lemma_Conjugates}--\ref{Lemma_Algorythm_Why_Not_Normal} formulated below, by simultaneous induction on $h$. All these statements obviously hold for swords of height 0 (i.e., for ordinary words), and we will assume that they hold true for all swords of height $h-1$.

We will also use the following corollary of Proposition~\ref{Prop_LongestCommonPrefix}:

\begin{Corollary}\label{Cor_Compare_Swords}
There is an algorithm that checks whether or not two given normal swords are equal.
\end{Corollary}

\begin{proof}
Let $\sigma,\tau$ be normal swords. Using the algorithm of Proposition~\ref{Prop_LongestCommonPrefix}, we find their longest common prefix $\rho$. Clearly, $\sigma=\tau$ if and only if  $\sigma=\rho$ and $\tau=\rho$, and the two latter equalities are equivalent to the equalities $|\sigma|=|\rho|$ and respectively $|\tau|=|\rho|$ in view of Lemma~\ref{LemmaSingProperties}.
\end{proof}

\subsection{Periodic normal swords}
We will need two auxiliary facts similar to well-known properties of ordinary words.

Given a sword $\tau$, a sword $\sigma$ is called $\tau$-\emph{periodic} if $\sigma$ is a power or an $\omega$-power of $\tau$.

\begin{Lemma}\label{Lemma_Conjugates}
Let $\sigma,\tau$ be non-empty normal swords of height at most $h$ and $\tau\sigma = \sigma\tau$. Then there exists a fully normal sword $\pi$ such that both $\tau$ and $\sigma$ are $\pi$-periodic. If, besides that, $\sigma$ and $\tau$ are fully normal, then $h(\sigma) = h(\tau) = h(\pi)$ and $\sigma=\pi^{n_1}$, $\tau=\pi^{n_2}$  for some positive integers $n_1,n_2$.
\end{Lemma}

\begin{proof}
By symmetry, we may assume that $|\sigma|\ge|\tau|$. Observe that since $\sigma\tau=\tau\sigma$, the sword $\tau$ is a prefix of $\sigma$ by item~2 of Lemma~\ref{LemmaSingProperties}. Therefore, there exists a sword $\sigma'$ such that $\sigma = \tau\sigma'$ whence $\tau\sigma'\tau =\sigma\tau =\tau\sigma=\tau^2\sigma'$. Since $\tau^2$ is a prefix of a normal sword, it is normal whence $\tau$ is fully normal. Further, by item~1 of Lemma~\ref{LemmaSingProperties}, we can cancel $\tau$ in the equality $\tau\sigma'\tau = \tau^2\sigma'$, thus getting $\sigma'\tau = \tau\sigma'$.

First consider the case $h(\sigma)>h(\tau)$. Then $h(\sigma')>h(\tau)$ whence $|\sigma'|>|\tau|$. Arguing as in the preceding paragraph, we see that $\tau$ is a prefix of $\sigma'$ and there exists a sword $\sigma''$ of height $h(\sigma)$ such that $\sigma''\tau=\tau\sigma''$, and so on. Therefore, for any positive integer $m$, the sword $\tau^m$ is a prefix of the sword $\sigma$. It is clear this is only possible if $\sigma$ has some $\omega$-power of $\tau$ as a prefix. Therefore, we may assume that $\sigma=\tau^{\omega+k}\sigma_1$ for some integer $k$ and some sword $\sigma_1$.

Suppose that $h(\sigma_1)=h(\sigma)$. Observe that $\tau^{\omega+k}\sigma_1\tau =\sigma\tau =\tau\sigma=\tau^{\omega+k+1}\sigma_1$. Using item~1 of Lemma~\ref{LemmaSingProperties}, we conclude that $\sigma_1\tau=\tau\sigma_1$, and we can repeat the same steps as before to obtain that $\sigma_1=\tau^{\omega+\ell}\sigma_2$ for some integer $\ell$ and some sword $\sigma_1$.  Then, however, the sword  $\sigma=\tau^{\omega+k}\tau^{\omega+\ell}\sigma_2$ would not be normal, a contradiction.

Hence $h(\sigma_1)<h(\sigma)$. If $\sigma_1$ is not empty, we can apply the induction hypothesis to the swords $\sigma_1$ and $\tau$ whose height is strictly less than $h$. There exists a fully normal sword $\pi$ such that both $\sigma_1$ and $\tau$ are $\pi$-periodic but then $\sigma = \tau^{\omega+k}\sigma_1$ could not be normal, a contradiction again. Thus, $\sigma_1$ is empty, whence $\sigma=\tau^{\omega+k}$ is an $\omega$-power of the fully normal sword $\tau$ and the first statement of the lemma holds true.

Assume now that $h(\sigma)=h(\tau)$. Then $|\sigma|=\ell_1\omega^h+g_1(w)$, $|\tau|=\ell_2\omega^h+g_2(\omega)$, where $\deg(g_i)<h$ for $i=1,2$. We proceed by induction on $\max\{\ell_1,\ell_2\}$.

If $\ell_1=\ell_2$ (in particular, if $\max\{\ell_1,\ell_2\}=1$), the sword $\sigma'$ such that $\sigma=\tau\sigma'=\sigma'\tau$ has height less than $h$, and this leads to the case just considered. We then obtain that there is a fully normal sword $\pi$ such that $\tau=\pi^{\omega+k}$ and $\sigma'=\pi^m$ for some integers $k,m$, and thus, $\sigma=\pi^{\omega+k+m}$. If $\ell_1\ne\ell_2$, we apply the induction hypothesis to the pair $\sigma',\tau$ to obtain the required property.

The second statement of the lemma trivially follows from the definition of a fully normal sword.
\end{proof}

\begin{Lemma}\label{Lemma_Periodic_Square}
Let $x$ and $y$ be fully normal swords. Suppose that $|x|\ge|y|$ and $x^2$ is a prefix \textup[suffix\textup] of some $y$-periodic sword. Then there exists a fully normal sword $z$ such that $x$ and $y$ are powers of $z$. If, besides that, some $\omega$-powers of $x$ and $y$ are prefixes \textup[suffixes\textup] of a normal sword, then $x = y$.
\end{Lemma}

\begin{proof}
Since $x^2$ is a prefix of a power or an $\omega$-power of $y$ and $|x|\ge|y|$, a power or an $\omega$-power of $y$  should be a prefix of $x$. In the latter case there are swords $y_1,y_2$ and a prefix $y_3$ of $y$ such that $y=y_1y_2$ and $x=(y_1y_2)^{\omega+r_1}y_1=y_2(y_1y_2)^{\omega+r_2}y_3$ for some $r_1,r_2\in\integers$ whence
\[
x^2=(y_1y_2)^{\omega +r_1}y_1y_2(y_1y_2)^{\omega+r_2}y_3=y_1(y_2y_1)^{\omega +r_1}(y_2y_1)^{\omega +r_2}y_1y_3.
\]
We see that $x$ is not fully normal, a contradiction.

If a power of $y$ but no $\omega$-power of $y$ is a prefix of $x$, then there exist swords $y_1,y_2$ and a prefix $y_3$ of $y$ such that $y_1y_2=y$  and $x=(y_1y_2)^{n_1}y_1=y_2(y_1y_2)^{n_2}y_3$ for some $n_1,n_2\ge1$. Since both $y_1y_2$ and $y_2y_1$ are prefixes of $x$ and have the same length, we conclude that $y_1y_2=y_2y_1$. Since $y$ is fully normal, by Lemma~\ref{Lemma_Conjugates} we obtain that $y_1 = z^{k_1}$, $y_2 =z^{k_2}$ for some $k_1,k_2\ge1$ and some fully normal sword $z$. Hence $x$ and $y$ are powers of $z$.

The second statement of the lemma immediately follows from the fact that, because of the rule~\eqref{eq:multiplication1}, the only power of $z$ that can occur as a factor in a normal sword is $z$ itself. Hence, $x=z$ and $y=z$.
\end{proof}

\subsection{Normalization}
Our algorithm that constructs a normal form of an arbitrary sword will repeatedly invoke two procedures.
Recall that every sword of height $h$ can be represented as $\pi_0\rho_1^{\omega+q_1}\pi_1\cdots\rho_n^{\omega+q_n}\pi_n$, where each of the swords $\pi_0,\pi_1,\dots,\pi_n,\rho_1,\dots,\rho_n$ has height less than $h$. Therefore it suffices to exhibit a procedure that produces a normal form for a given sword being an $\omega$-power of a normal swords and a procedure that produces a normal form for the product of two given normal swords.

We start with the latter procedure. First of all, observe that the product of two normal swords may indeed fail to be normal. For example, if $\sigma_1=(xz^{\omega+2})^{\omega}$ and $\sigma_2 = z^{\omega-5}$, then the product $\sigma_1\sigma_2=(xz^{\omega+2})^{\omega}z^{\omega-5}$ is not normal because it is equal (as a sword) to
$(xz^{\omega+2})^{\omega-1}xz^{\omega+2}z^{\omega-5}$. One can obtain a normal form of $(xz^{\omega+2})^{\omega-1}xz^{\omega+2}z^{\omega-5}$ by applying to it just one $\mathcal{R}/\mathcal{S}$-transition of the form \eqref{eq:addition}, thus producing $(xz^{\omega+2})^{\omega-1}xz^{\omega-3}$.

Actually, the situation demonstrated by the above example is generic. By the definition of a normal sword, it is clear that the product of two normal swords is not normal if and only if it has a factor of the form $y^{\omega+r}y^{\omega+q}$. Thus, the fact that $\sigma_1$ and $\sigma_2$ are normal while $\sigma_1\sigma_2$ is not is equivalent to the existence of a sword $y$ such that $\sigma_1$ has a suffix and $\sigma_2$ has a prefix that are $\omega$-powers of $y$. Obviously, if such a sword $y$ exists, it should be fully normal and by Lemma~\ref{Lemma_Periodic_Square} it is unique. We refer to $y$ as to the \emph{overlap} between $\sigma_1$ and $\sigma_2$.

\begin{Lemma}\label{Lemma_Why_Is_Not_Normal}
Let $\sigma_1,\sigma_2$ be two normal swords of height no more than $h$. There exists an algorithm that decides whether the product $\sigma_1\sigma_2$ is in normal form. If the product is not normal, the algorithm finds the overlap between $\sigma_1$ and $\sigma_2$ and constructs a normal form for the product.
\end{Lemma}

\begin{proof}
First, we describe an algorithm that, given two normal swords $\sigma_1,\sigma_2$ of height $h$, decides if they have an overlap of height $h-1$, and in the case where the answer is positive, finds the overlap and constructs a normal form for $\sigma_1\sigma_2$. Since the algorithm will be repeatedly invoked in the rest of the proof, it is convenient to give it a name. So, let us call this algorithm \textsf{TO} (from ``tall overlap'').

We start with representing the given swords as
\begin{equation}
\label{eq:present_for_product}
\sigma_1=\sigma'_1\rho_1^{\omega+q_1}\pi_1\ \text{ and }\ \sigma_2=\pi_2\rho_2^{\omega+q_2}\sigma'_2,
\end{equation}
where $h(\rho_1)=h(\rho_2)=h-1$ and $h(\pi_1), h(\pi_2)\le h-1$. If $\sigma_1$ has a suffix that is an $\omega$-power of some sword of height $h-1$, then the suffix is of height $h$, and hence, it has to involve some $\omega$-power of $\rho_1$ as a factor. Clearly, this is only possible if $\rho_1=\pi_1\rho'_1$ for some sword $\rho_1'$, and we can decide whether or not the latter equality holds true by applying the algorithm of Proposition~\ref{Prop_LongestCommonPrefix} to the swords $\pi_1$ and $\rho_1$ to find their longest common prefix and then deciding whether or not this prefix is equal to $\pi_1$. If the answer is positive, we can rewrite the sword $\sigma_1$ (by ``rolling'' $(\pi_1\rho'_1)^{\omega+q_1}$ to the right) as
\[
\sigma_1=\sigma'_1\pi_1(\rho'_1\pi_1)^{\omega+q_1}.
\]
Similarly, if $\sigma_2$ has a prefix that is an $\omega$-power of some sword of height $h-1$, then the prefix is of height $h$ and has some $\omega$-power of $\rho_2$ as a factor. This is only possible if $\rho_2=\rho'_2\pi_2$ for some sword $\rho_2'$, which can be verified by applying the ``dual'' of the algorithm of Proposition~\ref{Prop_LongestCommonPrefix}. If the answer is positive, we can rewrite the sword $\sigma_2$ (by ``rolling'' $(\rho'_2\pi_2)^{\omega+q_2}$ to the left) as
\[
\sigma_2=(\pi_2\rho'_2)^{\omega+q_2}\pi_2\sigma'_2.
\]
Now it remains to verify whether or not the swords $\rho'_1\pi_1$ and $\pi_2\rho'_2$ of height $h-1$ are equal. If the answer is negative, $\sigma_1$ and $\sigma_2$ have no overlap of height $h-1$. Otherwise, the sword $y=\rho'_1\pi_1=\pi_2\rho'_2$ is the overlap of height $h-1$ between $\sigma_1$ and $\sigma_2$ and $\sigma_1\sigma_2=\sigma'_1\pi_1y^{\omega+q_1}y^{\omega+q_2}\pi_2\sigma'_2$. Applying an $\mathcal{R}/\mathcal{S}$-transition of the form \eqref{eq:addition}, we reduce it to the sword $\tau=\sigma'_1\pi_1y^{\omega+q_1+q_2}\pi_2\sigma'_2$. We are going to prove that $\tau$ is normal, and thus, it is a normal form for the product $\sigma_1\sigma_2$. Arguing by contradiction, assume that $\tau$ is not normal. Then there should be the overlap $z$, say, between either $\sigma'_1\pi_1$ and $y^{\omega+q_1+q_2}\pi_2\sigma'_2$ or $\sigma'_1\pi_1y^{\omega+q_1+q_2}$ and $\pi_2\sigma'_2$. Since the height of $\tau$ is $h$, the height of $z$ is at most $h-1$ whence $z$ contains no $\omega$-power of $y$ as a prefix nor as a suffix. Then Lemma~\ref{Lemma_Periodic_Square} easily implies that $z=y$, but this leads to a contradiction: if $\sigma'_1\pi_1$ has an $\omega$-power of $y$ as a suffix, the sword $\sigma_1=\sigma'_1\pi_1y^{\omega+q_1}$ could not be normal, and if $\pi_2\sigma'_2$ has an $\omega$-power of $y$ as a prefix, the same conclusion applies to the sword $\sigma_2=y^{\omega+q_2}\pi_2\sigma'_2$.

Now assume that the algorithm \textsf{TO} has revealed that the swords $\sigma_1$ and $\sigma_2$ have no overlap of height $h-1$. Then we have to check whether they have an overlap of lesser height. For this, we again represent the swords in the form~\eqref{eq:present_for_product} and apply the algorithm of Lemma~\ref{Lemma_Why_Is_Not_Normal} to the swords $\rho_1^2\pi_1$ and $\pi_2\rho_2^2$, which are both of height less than $h$. If there is no overlap between $\rho_1^2\pi_1$ and $\pi_2\rho_2^2$, then there also is no overlap of height less than $h-1$ between $\sigma_1$ and $\sigma_2$ and the product $\sigma_1\sigma_2$ is in normal form. Otherwise, the algorithm finds the overlap between $\rho_1^2\pi_1$ and $\pi_2\rho_2^2$. Denoting the overlap by $y$, we can represent these swords as
\[
\rho_1^2\pi_1=\sigma''_1y^{\omega+r_1} \text{ and } \pi_2\rho_2^2=y^{\omega+r_2}\sigma''_2
\]
for some $r_1,r_2\in\integers$. Consider the sword $\theta=\sigma'_1\rho_1^{\omega+q_1-2}\sigma''_1y^{\omega+r_1+r_2}\sigma''_2\rho_2^{\omega+q_2-2}\sigma'_2$ obtained from $\sigma_1\sigma_2$ by an $\mathcal{R}/\mathcal{S}$-transition of the form \eqref{eq:addition}.  Clearly, to check whether $\theta$ is normal amounts to examining overlaps between the swords $\sigma'_1\rho_1^{\omega+q_1-2}\sigma''_1$ and $y^{\omega+r_1+r_2}\sigma''_2\rho_2^{\omega+q_2-2}\sigma'_2$ or between the swords $\sigma'_1\rho_1^{\omega+q_1-2}\sigma''_1y^{\omega+r_1+r_2}$ and $\sigma''_2\rho_2^{\omega+q_2-2}\sigma'_2$. In contrast to the situation in the previous paragraph, here the overlaps may exist, and it appears to make sense to illustrate this phenomenon by an example.

Thus, consider the normal swords $\tau_1=(x^{\omega}y)^{\omega+1}x^{\omega}$ and $\tau_2=(x^{\omega}y)^{\omega-2}$ of height 2. They have no overlap of height 1, but do have the overlap of height 0, namely, $x$: this becomes obvious if we rewrite $\tau_2$ as $x^{\omega}y(x^{\omega}y)^{\omega-3}$. Applying an $\mathcal{R}/\mathcal{S}$-transition of the form \eqref{eq:addition} to the product $\tau_1\tau_2$, we obtain the sword $(x^{\omega}y)^{\omega+1}x^{\omega}y(x^{\omega}y)^{\omega-3}$, and we see that the swords $(x^{\omega}y)^{\omega+1}$ and $x^{\omega}y(x^{\omega}y)^{\omega-3}=(x^{\omega}y)^{\omega-2}$ have the overlap $x^{\omega}y$, this time of height 1. Resolving the overlap with yet another transition of the form~\eqref{eq:addition}, we finally obtain the normal sword $(x^{\omega}y)^{\omega-1}$ as a normal form for $\tau_1\tau_2$.

Observe that the above example illustrates not only the difficulty we encounter but also a way to overcome it: since the new overlap has height $h(\tau_1)-1$, it can be found by the algorithm \textsf{TO}. Mutatis mutandis, this idea works also in the general case.

We return to the proof. Recall that we have to decide if there are overlaps between $\sigma'_1\rho_1^{\omega+q_1-2}\sigma''_1$ and $y^{\omega+r_1+r_2}\sigma''_2\rho_2^{\omega+q_2-2}\sigma'_2$ or between $\sigma'_1\rho_1^{\omega+q_1-2}\sigma''_1y^{\omega+r_1+r_2}$ and $\sigma''_2\rho_2^{\omega+q_2-2}\sigma'_2$. First, we apply \textsf{TO} to decide if there are overlaps of height $h-1$, and in the case where the answer is positive, to construct a normal form for the sword $\theta$, which will be also a normal form for $\sigma_1\sigma_2$. Suppose that \textsf{TO} has found no overlaps of height $h-1$. Then we apply the algorithm of Lemma~\ref{Lemma_Why_Is_Not_Normal} to the swords $\rho_1^2\sigma''_1$ and $y^{\omega+r_1+r_2}\sigma''_2\rho_2^2$ or to the swords $\rho_1^2\sigma''_1y^{\omega+r_1+r_2}$ and $\sigma''_2\rho_2^2$, which are all of height less than $h$. If no overlaps are found, then there are no overlaps of height less than $h-1$ between $\sigma'_1\rho_1^{\omega+q_1-2}\sigma''_1$ and $y^{\omega+r_1+r_2}\sigma''_2\rho_2^{\omega+q_2-2}\sigma'_2$ nor between $\sigma'_1\rho_1^{\omega+q_1-2}\sigma''_1y^{\omega+r_1+r_2}$ and $\sigma''_2\rho_2^{\omega+q_2-2}\sigma'_2$, and the sword $\theta$ is normal. Suppose that the algorithm has found the overlap $z$, say, between $\rho_1^2\sigma''_1$ and $y^{\omega+r_1+r_2}\sigma''_2\rho_2^2$ or between $\rho_1^2\sigma''_1y^{\omega+p_1+p_2}$ and $\sigma''_2\rho_2^2$. If $z$ has no $\omega$-power of $y$ as a prefix nor as a suffix, then Lemma~\ref{Lemma_Periodic_Square} implies that $z=y$, and this leads to a contradiction: if $\rho_1^2\sigma''_1$ has an $\omega$-power of $y$ as a suffix, the sword $\sigma_1=\sigma'_1\rho_1^{\omega+q_1-4}\rho_1^2\sigma''_1y^{\omega+r_1} $ could not be normal, and if $\sigma''_2\rho_2^2$ has an $\omega$-power of $y$ as a prefix, the same conclusion applies to the sword $\sigma_2=y^{\omega+r_2}\sigma''_2\rho_2^2\rho_2^{\omega+q_2-4}\sigma'_2\rule{0pt}{11pt}$. Thus, $z$ has no $\omega$-power of $y$ as a prefix or as a suffix whence $h(y)<h(z)<h-1$. Now we apply the same sequence of arguments to $z$, etc. Clearly,  each step of the described process deals with some sword $\theta$ obtained from $\sigma_1\sigma_2$ by a sequence of $\mathcal{R}/\mathcal{S}$-transitions and leads to exactly one of the three following results:
\begin{enumerate}
\item \textsf{TO} finds the overlap of height $h-1$, resolves it in $\theta$, and returns a normal form for $\sigma_1\sigma_2$.
\item \textsf{TO} finds no overlap of height $h-1$, and the algorithm of Lemma~\ref{Lemma_Why_Is_Not_Normal} applied to suitable words of height less than $h$ finds no overlap of height less than $h-1$, and therefore, $\theta$ constitutes a normal form for $\sigma_1\sigma_2$.
\item \textsf{TO} finds no overlap of height $h-1$, but the algorithm of Lemma~\ref{Lemma_Why_Is_Not_Normal} applied to suitable words of height less than $h$ finds the overlap whose height  is less than $h-1$ but greater than the height of the overlap found in the previous step. In this case we resolve the overlap in $\theta$ and pass to the next step.
\end{enumerate}
It is clear that the number of steps of our process does not exceed $h$, and therefore, it will eventually return
a normal form for $\sigma_1\sigma_2$.

It remains to analyze the situation where one of the swords $\sigma_1$ and $\sigma_1$ is of height less than $h$ while the other has height $h$. (In the case where both $h(\sigma_1)<h$ and $h(\sigma_2)<h$, the induction hypothesis applies immediately.) The argument is similar to that described above but simpler because there is no need to invoke the algorithm \textsf{TO}---one of the swords is of height less than $h$, no overlap of height $h-1$ is possible. By symmetry, we may suppose that $h(\sigma_1)<h$ and $h(\sigma_2)=h$. Then we can represent $\sigma_2$ as $\sigma_2=\pi_2\rho_2^{\omega+q_2}\sigma'_2$, where $h(\rho_2)=h-1$, $h(\pi_2)\leq h-1$. We apply the algorithm of Lemma~\ref{Lemma_Why_Is_Not_Normal} to the swords $\sigma_1$ and $\pi_2\rho_2^2$, which are both of height less than $h$. If there is no overlap between $\sigma_1$ and $\pi_2\rho_2^2$, then there also is no overlap between $\sigma_1$ and $\sigma_2$ and the product $\sigma_1\sigma_2$ is in normal form. Otherwise, the algorithm finds the overlap between $\sigma_1$ and $\pi_2\rho_2^2$. Denoting the overlap by $y$, we can represent these swords as
\[
\sigma_1=\sigma''_1y^{\omega+r_1} \text{ and } \pi_2\rho_2^2=y^{\omega+r_2}\sigma''_2
\]
for some $r_1,r_2\in\integers$. Consider the sword $\theta=\sigma''_1y^{\omega+r_1+r_2}\sigma''_2\rho_2^{\omega+q_2-2}\sigma'_2$ obtained from $\sigma_1\sigma_2$ by an $\mathcal{R}/\mathcal{S}$-transition of the form \eqref{eq:addition}. Clearly, to check whether $\theta$ is normal amounts to examining the overlaps between the swords $\sigma''_1$ and $y^{\omega+r_1+r_2}\sigma''_2\rho_2^{\omega+q_2-2}\sigma'_2$ and between the swords $\sigma''_1y^{\omega+r_1+r_2}$ and $\sigma''_2\rho_2^{\omega+q_2-2}\sigma'_2$. For this, we apply the algorithm of Lemma~\ref{Lemma_Why_Is_Not_Normal} to two pairs of swords: $\sigma''_1$ and $y^{\omega+r_1+r_2}\sigma''_2\rho_2^2$ and respectively $\sigma''_1y^{\omega+r_1+r_2}$ and $\sigma''_2\rho_2^2$, which swords are all of height less than $h$. If no overlaps have been discovered, the sword $\theta$ is normal and thus constitutes a normal form for the product $\sigma_1\sigma_2$. If the overlap $z$, say, has been found, then the same argument as above shows that $h(z)>h(y)$. We apply the same arguments to $z$, etc. The described process will eventually stop since the height of overlaps increases on each step but cannot reach $h-1$.
\end{proof}

Now we study how to construct a normal form for the $\omega$-power of a normal sword.
\begin{Lemma}\label{Lemma_Big_Rings}
Let $\rho$ be a normal sword of height $h-1$. There is an algorithm that reduces the sword $\rho^{\omega+q}$ either to a normal sword of height less than $h$ or  to a sword of the form $\alpha_1\beta^{\omega+t}\alpha_2$ in which $h(\alpha_1),h(\alpha_2)\leq h-1$, $h(\beta)=h-1$, and the sword $\beta^{\omega+t}$ is normal.
\end{Lemma}

\begin{proof}
Assume that $\rho^{\omega+q}$ is not normal. This means that an $\mathcal{R}/\mathcal{S}$-transition can be applied to the sword. Depending on the form of the transition, one of the three following cases occurs:
\begin{enumerate}
      \item[1)] $\rho = x^{\omega+m_1}x'x^{\omega+m_2}$ for some $m_1,m_2\in\integers$ and some non-empty $x'$ (the sword $x'$ cannot be empty since $\rho$ is normal);
      \item[2)] $\rho = x^n$ for some integer $n\ge 2$;
      \item[3)] $\rho = x^{\omega+m}$ for some integer $m$.
\end{enumerate}
We apply the algorithm of Lemma~\ref{Lemma_Algorythm_Why_Not_Normal} to the normal sword $\rho$ of height $h-1$ to check whether one of the cases 2) or 3) occurs, and if it does, to find a corresponding fully normal sword $x$. We may additionally assume that $x\ne y^p$ for any integer $p\ge 2$: in case 2) such an $x$ is the sword of minimum length returned by the algorithm  while in case 3) this holds automatically since $\rho$ is normal. Then $\rho^{\omega +q}$ can be reduced to $x^{\omega+nq}$ in case 2) or to $x^{\omega+mq}$ in case 3) and either of these two $\omega$-powers is easily seen to be normal.

In the remaining case 1), the sword $\rho$ must have a proper overlap with itself. We can check if this indeed happens by invoking the algorithm of Lemma~\ref{Lemma_Why_Is_Not_Normal} since the height of $\rho$ is less than $h$. Moreover, if the overlap exists, it is unique and fully  normal, and the algorithm finds it. Thus, assume that $\rho=x_0^{\omega+m_1}x_0'x_0^{\omega+m_2}$ for some fully normal sword $x_0$ and some $m_1,m_2\in\integers$. Then the sword $\rho^{\omega+q} = (x_0^{\omega+m_1}x_0'x_0^{\omega+m_2})^{\omega+q}$  can be reduced to the sword $x_0^{\omega+m_1} x_0'(x_0^{\omega+m_1+m_2}x_0')^{\omega+q-1}x_0^{\omega+m_2}$. Consider the sword  $\rho_1=x_0^{\omega+m_1+m_2}x_0'$. It is normal because any $\mathcal{R}/\mathcal{S}$-transition applicable to $\rho_1$ could have been applied also to $\rho$ while $\rho$ is normal. If $\rho_1^{\omega+q-1}$ is normal, we are done as we can put $\alpha_1=x_0^{\omega+m_1}x_0'$, $\alpha_2=x_0^{\omega+m_2}$, and $\beta=\rho_1$.

Suppose that $\rho_1^{\omega+q-1}$ is not normal. Then one of the above cases 1)--3) takes place for $\rho_1$.
We already have seen how to handle case 2): we can reduce $\rho_1^{\omega+q-1}$ to its normal form being an $\omega$-power of a fully normal word of height less than $h$. Substituting this normal form for $\rho_1^{\omega+q-1}$ in the sword $x_0^{\omega+m_1}x_0'\rho_1^{\omega+q-1}x_0^{\omega+m_2}$ produces a sword of the desired form $\alpha_1\beta^{\omega+t}\alpha_2$.

In case 3) there is a fully normal sword $x_1$ of height less than $h-1$ such that $\rho_1=x_1^{\omega+\ell}$ for some integer $\ell$. Then $x_0^{\omega+m_1}x_0'\rho_1^{\omega+q-1}x_0^{\omega+m_2}=
x_0^{\omega+m_1}x_0'(x_1^{\omega+\ell})^{\omega+q-1}x_0^{\omega+m_2}$, and we can apply an $\mathcal{R}/\mathcal{S}$-reduction of the form \eqref{eq:multiplication2} to obtain the sword $x_0^{\omega+m_1}x_0'x_1^{\omega+\ell(q-1)}x_0^{\omega+m_2}$ of height less than $h$. We can construct a normal form using of the latter sword using Proposition~\ref{Prop_Normalization}.

Finally, consider the situation where $\rho_1$ falls in case 1). Assume that $\rho_1=x_1^{\omega+\ell_1}x_1'x_1^{\omega +\ell_2}$ for some $\ell_1,\ell_2\in\integers$ and some non-empty $x_1'$. If $ h(x_0)=h(x_1)$, then Lemma~\ref{Lemma_Periodic_Square} implies $x_0=x_1$, and the sword $x_0'$ has some $\omega$-power of $x_0$ as a suffix. This would mean that the sword $\rho=x_0^{\omega+m_1}x_0'x_0^{\omega+m_2}$ is not normal, a contradiction. If $h(x_0)>h(x_1)$, then $x_0$ has a prefix of the form $x_1^{\omega+r_1}$ while  $x_0'$ has a suffix of the form $x_1^{\omega+r_2}$, and this again contradicts the assumption that $\rho$ is normal. Thus, we conclude that $h(x_0)<h(x_1)$.

We proceed in the same way by reducing the sword $\rho_1^{\omega+q-1}=(x_1^{\omega+\ell_1}x_1'x_1^{\omega+\ell_2})^{\omega+q-1}$ to the sword $x_1^{\omega+\ell_1}x_1'(x_1^{\omega+\ell_1+\ell_2}x_1')^{\omega+q-2}x_1^{\omega+\ell_2}$ and letting
$\rho_2=x_1^{\omega+\ell_1+\ell_2}x_1'$. The only situation where we cannot immediately normalize the sword $\rho_2^{\omega+q-2}$ is that where $\rho_2=x_2^{\omega+k_1}x_2'x_2^{\omega +k_2}$ for some $k_1,k_2\in\integers$ and some non-empty $x_2'$. In this situation, as the argument from the preceding paragraph shows, we must have $h(x_1)<h(x_2)$. We then repeat the procedure, if necessary, until it eventually stops at some sword $\rho_s$ such that the sword $\rho_s^{\omega+q-s}$ can be reduced to a normal sword of the form $\beta^{\omega+t}$  with $h(\beta)<h$. (Since the height of the overlaps $x_0,x_1,x_2,\dots$ increases on each step, the number of steps is upper bounded by $h$.) Then the original $\omega$-power $\rho^{\omega+q}$ is reduced to the sword $\alpha_1\beta^{\omega+t}\alpha_2$, where $\alpha_1=x_0^{\omega+m_1}x_0'x_1^{\omega+\ell_1}x_1'\cdots x_s^{\omega+r_1}x_s'$ and $\alpha_2=x_s^{\omega+r_2}\cdots x_1^{\omega+\ell_2}x_0^{\omega+m_2}$.
\end{proof}

Now we are in a position to prove Proposition~\ref{Prop_Normalization}. Recall that is claims the existence of an algorithm that, given a sword $\sigma$ of height $h$, returns a normal form of $\sigma$.

\begin{proof}[of Proposition~\ref{Prop_Normalization}]
Let $\sigma=\pi_0\rho_1^{\omega+q_1}\pi_1\cdots\rho_n^{\omega+q_n}\pi_n$ be a height representation of $\sigma$. By the induction hypothesis, we assume that all the swords $\rho_1,\dots,\rho_n$ and $\pi_0,\pi_1,\dots,\pi_n$ are normal.

By Lemma~\ref{Lemma_Big_Rings}, each sword $\rho_i^{\omega+q_i}$ can be effectively reduced to a sword of the form $\alpha_{i1}\beta_i^{\omega+r_i}\alpha_{i2}$ where the sword $\beta_i^{\omega+r_i}$ is normal of height $h$ while the swords $\alpha_{i1},\alpha_{i2}$ have height at most $h-1$ and again may be assumed to be normal by the induction hypothesis. Then $\sigma$ reduces to the following product of normal factors:
\[
\pi_0\alpha_{11}\beta_1^{\omega+r_1}\alpha_{12}\pi_1\dots\alpha_{n1}\beta_i^{\omega+r_i}\alpha_{n2}\pi_n.
\]
The algorithm of Lemma~\ref{Lemma_Why_Is_Not_Normal} allows us to construct a normal form for the product of two normal swords. Clearly, applying this algorithm several times, we can construct a normal form for the product of any finite number of normal swords, and hence, a normal form for $\sigma$.
\end{proof}

\subsection{Longest common prefix}
Here we prove Proposition~\ref{Prop_LongestCommonPrefix}. Recall that it claims the existence of an algorithm that, given two normal swords $\sigma$ and $\delta$, say, both of height at most $h$, returns their longest common prefix, which we denote by $LCP(\sigma,\delta)$.

\begin{proof}[of Proposition~\ref{Prop_LongestCommonPrefix}]
We fix a height representation of the sword $\sigma$ with the minimum number of factors of height $h$ and an analogous height representation of the sword $\delta$. We are going to induct on the total number $t$ of factors of height $h$ in these two representations. If $t=0$, then $h(\sigma),h(\delta)<h$ and the hypothesis of external induction (on $h$) applies. Consider the situation where one of the swords $\sigma$ and $\delta$ has height less than $h$ while the other has height $h$. For certainty, suppose that $h(\delta)<h(\sigma)=h$ and let $\pi_0\rho_1^{\omega+q_1}\cdots\rho_n^{\omega+q_n}\pi_n$ be the chosen height representation of the sword $\sigma$. Then $LCP(\sigma,\delta)=LCP(\pi_0\rho_1^k,\delta)$ where $k$ is any integer large enough to ensure that $|\pi_0\rho_1^k|>|\delta|$. Since $h(\pi_0\rho_1^k)<h$, we can construct $LCP(\pi_0\rho_1^k,\delta)$ using the external induction hypothesis. Therefore, we may assume that both $\sigma$ and $\delta$ have height $h$.

First consider a special case where $\sigma=\pi\rho^{\omega+q}$ and $\delta=\alpha\beta^{\omega+r}$ for some $q,r\in\integers$, some swords $\rho,\beta$ of height $h-1$ and some swords $\pi,\alpha$ of height at most $h-1$. By symmetry, we may assume that $|\alpha|\le|\pi|$.

We start with finding $\tau_1=LCP(\alpha, \pi)$ using the external induction hypothesis. If $|\tau_1|< |\alpha|$, then clearly $LCP(\sigma,\delta)=\tau_1$. Otherwise, $\tau_1=\alpha$ and $\pi=\tau_1\pi'$ for some sword $\pi'$ of height less than $h$. Now we should find the longest common prefix $\tau_2$ of the swords $\beta^{\omega+r}$ and $\pi'\rho^{\omega+q}$. First, we verify whether $\tau_2$ can be ``long'', that is, whether its length can exceed $\max\{|\beta^2|,|\pi'\rho^2|\}$. Since $\tau_2$ is a prefix of $\beta^{\omega+r}$, we conclude that $\pi'=\beta^k\beta_1$ for some integer $k\ge0$ and some prefix $\beta_1$ of $\beta$. Let $\beta_2$ be such that $\beta=\beta_1\beta_2$. Then $\beta^{\omega+r}=\pi'\beta_2\beta^{\omega+r-k-1}=\pi'(\beta_2\beta_1)^{\omega+r-k-1}\beta_2$, and  Lemma~\ref{Lemma_Periodic_Square} implies that $\rho=\beta_2\beta_1$, see Fig.~5.
\begin{center}
\unitlength=.7mm
\begin{picture}(180,35)(0,-5)
\gasset{AHnb=0}
\drawline[linewidth=.5](5,7.5)(175,7.5)
\put(178,7.5){\dots}
\put(168,17.5){\dots}
\put(153,2.5){\dots}
\drawline(5,0)(5,25)
\drawline(45,7.5)(45,25)
\drawline(85,7.5)(85,25)
\drawline(125,7.5)(125,25)
\drawline(165,7.5)(165,25)
\drawline(30,7.5)(30,15)
\drawline(70,0)(70,15)
\drawline(110,0)(110,15)
\drawline(150,0)(150,15)
\gasset{AHnb=1}
\drawline(15,22.5)(45,22.5)
\drawline(15,22.5)(5,22.5)
\drawline(55,22.5)(45,22.5)
\drawline(55,22.5)(85,22.5)
\drawline(95,22.5)(85,22.5)
\drawline(95,22.5)(125,22.5)
\drawline(135,22.5)(125,22.5)
\drawline(135,22.5)(165,22.5)
\drawline(15,12.5)(30,12.5)
\drawline(15,12.5)(5,12.5)
\drawline(35,12.5)(30,12.5)
\drawline(35,12.5)(45,12.5)
\drawline(55,12.5)(70,12.5)
\drawline(55,12.5)(45,12.5)
\drawline(75,12.5)(70,12.5)
\drawline(75,12.5)(85,12.5)
\drawline(95,12.5)(110,12.5)
\drawline(95,12.5)(85,12.5)
\drawline(115,12.5)(110,12.5)
\drawline(115,12.5)(125,12.5)
\drawline(135,12.5)(150,12.5)
\drawline(135,12.5)(125,12.5)
\drawline(155,12.5)(150,12.5)
\drawline(155,12.5)(165,12.5)
\drawline(90,2.5)(70,2.5)
\drawline(90,2.5)(110,2.5)
\drawline(21,2.5)(5,2.5)
\drawline(21,2.5)(70,2.5)
\drawline(130,2.5)(150,2.5)
\drawline(130,2.5)(110,2.5)
\put(24.5,25){$\beta$}
\put(64.5,25){$\beta$}
\put(104.5,25){$\beta$}
\put(144.5,25){$\beta$}
\put(16,15){$\beta_1$}
\put(56,15){$\beta_1$}
\put(96,15){$\beta_1$}
\put(136,15){$\beta_1$}
\put(36,15){$\beta_2$}
\put(76,15){$\beta_2$}
\put(116,15){$\beta_2$}
\put(156,15){$\beta_2$}
\put(129.5,-2){$\rho$}
\put(38.5,-2){$\pi'$}
\put(89.5,-2){$\rho$}
\end{picture}\\
Figure 5: Structure of $\tau_2$ in case $|\tau_2|>\max\{|\beta^2|,|\pi'\rho^2|\}$
\end{center}
Hence, $\pi'\rho=\beta^k\beta_1\beta_2\beta_1=\beta\beta^k\beta_1=\beta\pi'$. Given the swords $\pi',\rho,\beta$, all of height less than $h$, we can verify whether the equality $\pi'\rho=\beta\pi'$ holds in view of  Corollary~\ref{Cor_Compare_Swords}. If the equality holds true, then clearly $\pi'\rho^\omega=\beta^\omega\pi'$ whence $\tau_2$ is the shortest of the swords $\pi'\rho^{\omega+q}$ and $\beta^{\omega+r}$. If the equality fails,
we conclude that $|\tau_2|\le\max\{|\beta^2|,|\pi'\rho^2|\}$ and by the external induction assumption we can find $\tau_2$ as $LCP(\beta^{\ell},\pi'\rho^{\ell})$ where $\ell$ is large enough to ensure that $\min\{|\beta^{\ell}|,|\pi'\rho^{\ell}|\}>\max(|\beta^2|,|\pi'\rho^2|)$. Finally, we obtain $LCP(\sigma,\delta)=\tau_1\tau_2$.

Now consider the general case. Isolating the leftmost factors of height $h$ in the height representations of the swords $\sigma$ and $\delta$ chosen at the beginning of the proof, we can represent these swords as follows:
\[
\sigma=\pi\rho^{\omega+q}\sigma', \quad \delta=\alpha\beta^{\omega+r}\delta',
\]
where $q,r\in\integers$, $h(\rho)=h(\beta)=h-1$, $h(\alpha),h(\pi)\leq h-1$. Using the algorithm for the special case that we described above, we can construct $\gamma_1=LCP(\pi\rho^{\omega+q},\alpha\beta^{\omega+r})$. If $h(\gamma_1)<h$, then clearly $LCP(\sigma,\delta)=\gamma_1$. Otherwise, the construction guarantees that $\gamma_1$ is the shortest of the swords $\pi\rho^{\omega+q}$ and $\alpha\beta^{\omega+r}$. For certainty, suppose that $\gamma_1=\pi\rho^{\omega+q}$ and $\alpha\beta^{\omega+r}=\gamma_1\delta''$. Since the total number of factors of height $h$ in some height representations of the swords $\sigma'$ and $\delta''\delta'$ is smaller than $t$, we can apply the internal induction assumption to find  $\gamma_2=LCP(\sigma',\delta''\delta')$. Then $LCP(\sigma,\delta)=\gamma_1\gamma_2$.
\end{proof}

In order to close the cycle of simultaneous induction, it remains to prove the following result:
\begin{Lemma}\label{Lemma_Algorythm_Why_Not_Normal}
There exists an algorithm that, given a normal sword $\sigma$ of height $h$, decides whether or not there exists a fully normal sword $x$ such that either $\sigma$ is an $\omega$-power of $x$ or there is an integer $m\ge2$ such that $\sigma=x^m$. If the answer is positive, the algorithm finds $x$ in the first case and all possible pairs $(x,m)$ in the second case.
\end{Lemma}

\begin{proof}
If $\sigma=x^{\omega+q}$ for some sword $x$ and integer $q$, then $|\sigma|=(\omega+q)|x|$ whence the number $-q$ must be a root of the polynomial $|\sigma|$. Thus, we can find all integer roots $q_1,\dots,q_n$ of $|\sigma|$ and  for each $i=1,\dots,n$, apply Lemma~\ref{Lemma_Find_Prefix_of_certain_length} to verify if $\sigma$ has a prefix $x_i$ of length $|\sigma|/(\omega+q_i)$. If this is the case, it remains to verify whether $x_i^{\omega+q_i}=\sigma$, and this can be done in view of Corollary~\ref{Cor_Compare_Swords}.

Similarly, $\sigma=x^m$ for some sword $x$ and integer $m\ge2$, then the number $m$ divides all coefficients of the polynomial $|\sigma|$. Thus, we can find all common divisors $m_1,\dots,m_k$ of the coefficients of $|\sigma|$
such that $m_1,\dots,m_k\ge2$. Then for each $j=1,\dots,k$, we apply Lemma~\ref{Lemma_Find_Prefix_of_certain_length} to check if $\sigma$ has a prefix $x_j$ of length $|\sigma|/m_j$, and if this is the case, we use Corollary~\ref{Cor_Compare_Swords} to verify whether $x_j^{m_j}=\sigma$.
\end{proof}

\section{Proof of the theorem}
\label{sec:Proof of the theorem}
\subsection{Normal swords and epigroup identities}
Recall that we consider epigroups as unary semigroups under the unary operation of pseudoinversion $x\mapsto\ov{x}$. It is known~\cite{Shevrin_survey} that the following identities hold in the class $\mathfrak{E}$ of all epigroups:
  \begin{gather}~\label{xyx}
   \ov{(xy)}x\is x\ov{(yx)},\\
    \label{x'2x=x'}
   \ov{x}^2x\is \ov{x},\\
   \label{x2x'=x''}
    x^2\ov{x}\is\ov{\ov{x}},\\
  \label{(x'x)'=x'x}
    \ov{\ov{x}x}\is\ov{x}x,\\
  \label{xk'=x'k}
    \ov{x^p}\is \ov{x}^p \hbox to 0mm{\quad\text{for each prime $p$}.}
  \end{gather}
We will often use also the identity $\ov{x}x\is x\ov{x}$ that can be easily deduced from \eqref{xyx}--\eqref{xk'=x'k}. Indeed,
\[
\ov{x}x\ \stackrel{\eqref{x'2x=x'}}{\is}\ \ov{x}^2x\cdot x\ \stackrel{\eqref{xk'=x'k}}{\is}\ \ov{x^2}x\cdot x\
\stackrel{\eqref{xyx}}{\is}\ x\ov{x^2}\cdot x=x\cdot \ov{x^2}x\ \stackrel{\eqref{xk'=x'k}}{\is}\ x\cdot\ov{x}^2x\ \stackrel{\eqref{x'2x=x'}}{\is}\ x\ov{x}.
\]

Every unary semigroup term becomes a $\mathbb{Z}$-unary word if one replaces every expression of the form $\ov{x}$ by $x^{\omega-1}$. Conversely, every $\mathbb{Z}$-unary word will be treated as an epigroup term in which expressions of the form $x^{\omega+q}$, $q\in\integers$, are interpreted as follows:
\[
x^{\omega+q}=
\begin{cases}
    \ov{x}x^{q+1}, & \hbox{ if } q \hbox{ is non-negative;} \\
    \ov{x}^{(-q)}, & \hbox{ otherwise.}
\end{cases}
\]

\begin{Proposition}\label{Prop_Identities_in_Swords}
If\/ $\mathbb{Z}$-unary words $\sigma,\sigma'$ are $\mathcal{S}$-related, then the identity $\sigma\is\sigma'$ holds in $\mathfrak{E}$.
\end{Proposition}

\begin{proof}
It suffices to verify that, for every integer $q$, each epigroup satisfies the identities
\begin{gather}
\label{eq:3 identities_a}
xx^{\omega+q}\is x^{\omega+q+1},\\
\label{eq:3 identities_b}
x^{\omega+q}x\is x^{\omega+q+1},\\
\label{eq:3 identities_c}
(xy)^{\omega+q} x\is x(yx)^{\omega+q}
\end{gather}
that correspond to the pairs \eqref{eq:wind1}--\eqref{eq:roll} generating the relation $\mathcal{S}$ as a fully invariant congruence.

First, suppose that $q\ge0$. Then $x^{\omega+q}=\ov{x}x^{q+1}$. Note that, since $\ov{x}x\is x\ov{x}$, we have $\ov{x}x^{q+1}\is x^{q+1}\ov{x}$. Therefore, the identities \eqref{eq:3 identities_a} and \eqref{eq:3 identities_b} clearly hold in $\mathfrak{E}$ since
\[
xx^{\omega+q}\is xx^{q+1}\ov{x}=x^{q+2}\ov{x}\is x^{\omega+q+1}\ \text{ and }\
x^{\omega+q}x=\ov{x}x^{q+1}x=\ov{x}x^{q+2}=x^{\omega+q+1}.
\]
Consider the identity  \eqref{eq:3 identities_c}. We have
\[
(xy)^{\omega+q}x=\ov{xy}(xy)^{q+1}x=\ov{xy}x(yx)^{q+1}\stackrel{\eqref{xyx}}{\is} x\ov{yx}(yx)^{q+1}=x(yx)^{\omega+q}.
\]

Now suppose that $q<0$. In this case $x^{\omega+q}=\ov{x}^{(-q)}$. If $q=-1$, then the identity \eqref{eq:3 identities_a} becomes $xx^{\omega-1}\is x^{\omega}$, and it clearly holds in $\mathfrak{E}$ because $xx^{\omega-1}=x\ov{x}$ while $x^\omega=\ov{x}x$. For $q<-1$, we use~\eqref{x'2x=x'} to obtain
\[
xx^{\omega+q}=x\ov{x}^{(-q)}=x\ov{x}^2\ov{x}^{(-q-2)}\stackrel{\eqref{x'2x=x'}}{\is}\ov{x}\,\ov{x}^{(-q-2)}= \ov{x}^{(-q-1)}=x^{\omega+q+1}.
\]
The identity \eqref{eq:3 identities_b} is treated in the same way. Finally, we have
\begin{multline*}
(xy)^{\omega+q}x=\ov{xy}^{(-q)}x=\ov{xy}^{(-q-1)}\ov{xy}x\stackrel{\eqref{xyx}}{\is} \ov{xy}^{(-q-1)}x\ov{yx}\stackrel{\eqref{xyx}}{\is}\dots\\
\stackrel{\eqref{xyx}}{\is}x\ov{yx}\,\ov{yx}^{(-q-1)}= x\ov{yx}^{(-q)}=x(yx)^{\omega+q},
\end{multline*}
thus establishing \eqref{eq:3 identities_c}.
\end{proof}

Proposition~\ref{Prop_Identities_in_Swords} allows us to treat expressions of the form $\sigma\is\sigma'$, with $\sigma,\sigma'$ being swords, as epigroup identities.

\begin{Proposition}\label{Prop_Nornal_Equalities}
Let $\sigma_1$ and $\sigma_2$ be two swords and let $\alpha_1$ and respectively $\alpha_2$ be their normal forms. Then the identity $\sigma_1\is\sigma_2$ holds in $\mathfrak{E}$ if and only if so does the identity $\alpha_1\is\alpha_2$.
\end{Proposition}

\begin{proof}
It is enough to show  that if $\alpha$ is a normal form of a sword $\sigma$, then the identity $\sigma\is\alpha$ holds in $\mathfrak{E}$. For this, it clearly suffices to prove that $\mathfrak{E}$ satisfies the identities
\begin{gather}
\label{eq:3 next identities_a}
x^{\omega+r}x^{\omega+q}\is x^{\omega+r+q},\\
\label{eq:3 next identities_b}
(x^n)^{\omega+q}\is x^{\omega+nq},\\
\label{eq:3 next identities_c}
(x^{\omega+r})^{\omega+q}\is x^{\omega+rq}
\end{gather}
for all $q,r\in\integers$ and all $n\ge2$ as these identities correspond to the reduction rules \eqref{eq:addition}--\eqref{eq:multiplication2} used in the normalization procedure.

The proof of~\eqref{eq:3 next identities_a} splits in 4 cases depending on the signs of $q$ and $r$. If $q,r\ge0$, we have
\[
x^{\omega+r}x^{\omega+q}=\ov{x}x^{r+1}\ov{x}x^{q+1}\is \ov{x}^2x^{r+q+2}\ \stackrel{\eqref{x'2x=x'}}{\is}\ \ov{x}x^{p+q+1}=x^{\omega+p+q}
\]
If $q,r<0$, the proof is straightforward since
\[
x^{\omega+r}x^{\omega+q}=\ov{x}^{(-r)}\,\ov{x}^{(-q)}=\ov{x}^{(-r-q)}=x^{\omega+r+q}.
\]

Assume that $r\ge0$ and $q<0$. Then
\[
x^{\omega+r}x^{\omega+q}=\ov{x}x^{r+1}\ov{x}^{(-q)}\is\ov{x}^{(-q+1)}x^{r+1}\ \stackrel{\eqref{x'2x=x'}}{\is}\
\begin{cases}
\ov{x}x^{r+q+1} &\text{ if } r\ge-q;\\
\ov{x^{-q-r}}   &\text{ if } r<-q.
\end{cases}
\]
On the other hand, by the definition, we have
\[
x^{\omega+r+q}=
\begin{cases}
\ov{x}x^{r+q+1} &\text{ if } r\ge-q;\\
\ov{x^{-q-r}}   &\text{ if } r<-q.
\end{cases}
\]
Thus, we conclude that $x^{\omega+r}x^{\omega+q}\is x^{\omega+r+q}$. The case where $r<0$ and $q\ge0$ is completely analogous.

Before we proceed with proving \eqref{eq:3 next identities_b}, we notice that the identity $\ov{x^n}\is\ov{x}^n$ holds in $\mathfrak{E}$ for every positive integer $n$. For this, we induct on $n$. The claim is trivial for $n=1$. If $n>1$, represent $n$ as $n=kp$, with $p$ being prime. Then we have $\ov{x^k}\is\ov{x}^k$ by the induction assumption whence
\[
\ov{x^n}=\ov{x^{kp}}=\ov{(x^k)^p}\ \stackrel{\eqref{xk'=x'k}}{\is}\ \ov{x^k}^p\is(\ov{x}^k)^p=\ov{x}^{kp}=\ov{x}^n. \]

Now, for each $q\ge0$, we have
\[
(x^n)^{\omega+q}=\ov{x^n}(x^n)^{q+1}\is \ov{x}^nx^{nq+n}\ \stackrel{\eqref{x'2x=x'}}{\is}\
\ov{x}^{n-1}x^{nq+n-1}\ \stackrel{\eqref{x'2x=x'}}{\is}\dots\ \stackrel{\eqref{x'2x=x'}}{\is}\ \ov{x}x^{nq+1}=x^{\omega+nq}.
\]
For $q<0$, we have
\[
(x^n)^{\omega+q}=\ov{x^n}^{(-q)}\is(\ov{x}^n)^{-q}=\ov{x}^{(-nq)}=x^{\omega+nq}.
\]

In order to verify \eqref{eq:3 next identities_c}, we first show that $\ov{\ov{\ov{x}}}\is\ov{x}$. Indeed, substituting $\ov{x}$ for $x$ in~\eqref{x2x'=x''}, we get
\[
\ov{\ov{\ov{x}}}\is\ov{x}^2\ov{\ov{x}}\ \stackrel{\eqref{x2x'=x''}}{\is}\ \ov{x}^2x^2\ov{x}\is \ov{x}^3x^2\ \stackrel{\eqref{x'2x=x'}}{\is}\ \ov{x}^2x\ \stackrel{\eqref{x'2x=x'}}{\is}\ \ov{x}.
\]
We use this to deduce yet another auxiliary identity, namely, $\ov{x^{\omega+r}}\is x^{\omega-r}$ for every integer $r\in\integers$. If $r>0$, we have
\begin{multline*}
\ov{x^{\omega+r}}=\ov{\ov{x}x^{r+1}}\ \stackrel{\eqref{x'2x=x'}}{\is}\ \ov{\ov{x}^2x^{r+2}}\ \stackrel{\eqref{x'2x=x'}}{\is}\dots\ \stackrel{\eqref{x'2x=x'}}{\is}\ \ov{\ov{x}^rx^{2r}}\is\ov{(\ov{x}x^2)^r}\is
\left(\ov{\ov{x}x^2}\right)^r\is\\
\left(\ov{x^2\ov{x}}\right)^r\ \stackrel{\eqref{x2x'=x''}}{\is}\ \left(\ov{\ov{\ov{x}}}\right)^r\is\ov{x}^r=(x^{\omega-1})^r\is x^{\omega-r}.
\end{multline*}
Here we have repeatedly employed \eqref{eq:3 next identities_a} in the final step. If $r=0$, we have
\[
\ov{x^\omega}=\ov{\ov{x}x}\ \stackrel{\eqref{(x'x)'=x'x}}{\is}\ \ov{x}x=x^\omega.
\]
If $r<0$, we have
\[
\ov{x^{\omega+r}}=\ov{\ov{x}^{(-r)}}\is\ov{\ov{x}}^{(-r)}\ \stackrel{\eqref{x2x'=x''}}{\is}\ (x^2\ov{x})^{-r}\is (\ov{x}x^2)^{-r}=\ov{x}^{-r}x^{-2r}\stackrel{\eqref{x'2x=x'}}{\is}\dots\ \stackrel{\eqref{x'2x=x'}}{\is}\ \ov{x}x^{-r+1}= x^{\omega-r}.
\]

Now we are ready to deduce the identity~\eqref{eq:3 next identities_c}. If $q\ge0$, we have
\[
(x^{\omega+r})^{\omega+q}=\ov{x^{\omega+r}}(x^{\omega+r})^{q+1}\is x^{\omega-r}(x^{\omega+r})^{q+1}\is x^{\omega+rq},
\]
and if $q<0$, we have
\[
(x^{\omega+r})^{\omega+q}\is \ov{x^{\omega+r}}^{(-q)}\is(x^{\omega-r})^{(-q)}\is x^{\omega+rq},
 \]
where the final steps in each of the two deductions follow from \eqref{eq:3 next identities_a}.
\end{proof}

\subsection{Connections to Burnside varieties}
For any pair $(m,k)$ of positive integers, the semigroup variety $\mathfrak{B}_{m,k}$ defined by the identity $x^{m}\is x^{m+k}$ is called a \emph{Burnside variety}. Let $A$ be a finite set and denote by $B(A, m, k)$ the free $A$-generated semigroup in the variety $\mathfrak{B}_{m,k}$. We recall some results due to Victor Guba~\cite{Guba1,Guba2}.

\begin{Theorem}[{\mdseries\cite{Guba2}, Theorem A}]
\label{THM_Burnside_Decidable}
For $m\ge3$, $k\ge1$, the word problem for the semigroup $B(A, m, k)$ is decidable.
\end{Theorem}

\begin{Theorem}
\label{THM_Burnside_Factors}
For $m\ge3$, $k\ge1$, every element of the semigroup $B(A, m, k)$ has a finite number of factors.
\end{Theorem}

The next statement easily follows from Theorem~\ref{THM_Burnside_Factors}. We supply a proof for the sake of completeness.
\begin{Proposition}
\label{Prop_Finite_Members_Burnside}
For $m\ge3$, $k\ge 1$, the Burnside variety $\mathfrak{B}_{m,k}$ is generated by its finite members.
\end{Proposition}

\begin{proof}
Take any two different elements $a,b\in B(A, m, k)$ and let $F_{ab}$ the set of all factors of their product $ab$. By Theorem~\ref{THM_Burnside_Factors}, $F_{ab}$ is finite, and it is easy to see that the set $I(ab)=B(A,m,k)\setminus F_{ab}$ forms an ideal of $B(A, m, k)$. Thus, the Rees quotient $B(A,m,k)/I(ab)$ is finite and the natural homomorphism  $B(A, m, k)\to B(A,m,k)/I(ab)$ separates $a$ and $b$ since $a,b\notin I(ab)$. Therefore, the semigroup $B(A, m, k)$ is residually finite for any finite set $A$. Since each variety is generated by the collection of its finitely generated free objects, the variety $\mathfrak{B}_{m,k}$ is generated by its finite members.
\end{proof}

In the sequel, it will be convenient to consider only the varieties $\mathfrak{B}_{m,k}$ with $k>m$.
Observe that every such variety satisfies the identity $x^k\is x^{2k}$. Indeed,
\[
x^{2k}=x^{(m+k)+(k-m)}=x^{m+k}x^{k-m}\is x^mx^{k-m}=x^{m+(k-m)}=x^k.
\]
Semigroups in the Burnside varieties are periodic, and hence, we may (and will) treat them as epigroups. It is easy to see that in every epigroup from $\mathfrak{B}_{m,k}$ with $k>m$ one has $\ov{x}=x^{k-1}$ and $x^{\omega}=\ov{x}x= x^k$.

Let $\alpha$ be a $\mathbb{Z}$-unary word. We define the \emph{depth} of $\alpha$ as follows:
\[
d(\alpha):=\max\{0,-q\mid \text{ the expression } \omega+q \text{ occurs in } \alpha \}.
\]
For instance, $d\left((x^{\omega-4}yx^{\omega+30})^{\omega-1}xy^{\omega}\right)=4$. Given an integer $k>d(\alpha)$, we denote by $\alpha^{(k)}$ the expression obtained from $\alpha$ by substituting $k$ for each occurrence of $\omega$ in $\alpha$. The inequality $k>d(\alpha)$ ensures that after the substitution, all expressions $\omega+q$ involved in $\alpha$ become positive integers, and hence, $\alpha^{(k)}$ is in fact a well-formed ordinary word. Now take any $\ell>d(\alpha)$, any $k>\ell$ and let $m=k-\ell$. Then it is easy to see that the identity $\alpha\is\alpha^{(k)}$ holds true in the variety $\mathfrak{B}_{m,k}$.

We want to extend ``ordinarization'' transforms of the form $\alpha\mapsto\alpha^{(k)}$ to swords. There is a subtlety here because the concept of depth is well defined only for $\mathbb{Z}$-unary words and not for swords: if two $\mathbb{Z}$-unary words $\alpha$ and $\alpha'$ are obtained from each other by ``unwinding'' or ``winding up'', they represent the same sword, but the equality $d(\alpha)=d(\alpha')$ need not be true. Therefore, given a sword $\sigma$, we should first fix a $\mathbb{Z}$-unary word $\alpha$ representing $\sigma$ and choose an integer $k$ strictly larger than $d(\alpha)$; we then denote by $\sigma^{(k)}$ the word $\alpha^{(k)}$. Observe that the conclusion that the identity $\sigma\is\sigma^{(k)}$ holds in the variety $\mathfrak{B}_{m,k}$ remains valid for \textbf{every} choice of a $\mathbb{Z}$-unary word representing the sword $\sigma$ provided that $k$ is strictly larger than $m$ and $\ell=k-m$ is is strictly larger than the depth of the chosen $\mathbb{Z}$-unary word.

We need also yet another auxiliary construction $\sigma\mapsto\wh{\sigma}$ that associates ordinary words to swords. Again, it is first defined for $\mathbb{Z}$-unary words and then we extend it to swords by utilizing the same approach as above, that is, by fixing a $\mathbb{Z}$-unary word representing a given sword. If $\alpha$ is a $\mathbb{Z}$-unary word, we define $\wh{\alpha}$ using induction on height. If $h(\alpha)=0$, i.e., if $\alpha$ is an ordinary word, then $\wh{\alpha}:=\alpha$. If $h(\alpha)>0$, let $\pi_0\rho_1^{\omega+q_1}\pi_1\cdots\rho_n^{\omega+q_n}\pi_n$ be the height representation of $\alpha$. Then for each $i=1,\dots,n$, we define $q'_i=2$ if $q_i\le2$ and $q'_i=q_i$ otherwise, and let $\wh{\alpha}:= \wh{\pi_0}\wh{\rho_1}^{q'_1}\wh{\pi_1}\cdots \wh{\rho_n}^{q'_n}\wh{\pi_n}$. For instance, if $\sigma=(x^{\omega-4}yx^{\omega+30})^{\omega-1}xy^{\omega}$, then $\wh{\sigma}=(x^2yx^{30})^2yx^2$.

In the following statements and their proofs, whenever both constructions $\sigma^{(k)}$ and $\wh{\sigma}$ are used, we assume that they are produced from the same $\mathbb{Z}$-unary word representing the sword $\sigma$.

\begin{Lemma}\label{Lemma_Equality_to_Burnside}
Let $\sigma_1,\sigma_2$ be two normal swords. Suppose that $\sigma_1^{(k)}=\sigma_2^{(k)}$ for every choice of\/ $\mathbb{Z}$-unary words representing $\sigma_1$ and $\sigma_2$, every $\ell>\max\{d(\sigma_1),d(\sigma_2)\}$ and every $k$ such that $k-\ell-2>\max\{|\wh{\sigma_1}|,|\wh{\sigma_2}|\}$. Then $\sigma_1=\sigma_2$.
\end{Lemma}

\begin{proof}
Suppose that $\sigma_1\ne\sigma_2$. Let $\tau$ be the longest common prefix of $\sigma_1$ and $\sigma_2$ (it exists by Proposition~\ref{Prop_LongestCommonPrefix}). We write $\sigma_1$ and $\sigma_2$ as $\sigma_1=\tau\rho_1$ and $\sigma_2=\tau\rho_2$ for some swords $\rho_1$ and $\rho_2$ such that either one of them is empty while the other is not or they both are non-empty and start with different letters. Consider some $\mathbb{Z}$-unary words $\alpha,\beta_1,\beta_2$ representing the swords $\tau,\rho_1,\rho_2$ respectively. Then, clearly, the $\mathbb{Z}$-unary words $\alpha\beta_1$ and $\alpha\beta_2$
represent the swords $\sigma_1$ and respectively $\sigma_2$. If the parameters $\ell$ and $k$ are calculated from these representations, then we have $\sigma^{(k)}_1=\tau^{(k)}\rho^{(k)}_1$ and $\sigma^{(k)}_2=\tau^{(k)}\rho^{(k)}_2$ whence $\rho^{(k)}_1=\rho^{(k)}_2$. Thus, the swords $\rho_1$ and $\rho_2$ that differ at their start should become equal when each $\omega$ in them is substituted with $k$. This is obviously not possible one of the swords is empty while the other is not since no ordinarization transform sends a non-empty sword to the empty word. Nor is it possible if $\rho_1$ and $\rho_2$ are non-empty and start with different letters because the starting letters of $\rho^{(k)}_1$ and $\rho^{(k)}_2$ remain the same as the ones of $\rho_1$ and respectively $\rho_2$. Thus, we arrive at a contradiction.
\end{proof}

Recall that an ordinary word $u$ is said to be \emph{periodic} if $u=v^n$ for some word $v$ and some $n\ge2$.
\begin{Lemma}\label{Lemma_Periodic_To_Burnside}
Let $\sigma$ be a normal sword. It is periodic if and only if the word $\sigma^{(k)}$ is periodic for every choice of a\/ $\mathbb{Z}$-unary word representing $\sigma$, every $\ell>d(\sigma)$ and every $k$ such that $k-\ell-2>|\wh{\sigma}|$.
\end{Lemma}

\begin{proof}
If $\sigma$ is a power then so is $\sigma^{(k)}$, and if $\sigma$ is an $\omega$-power, say, $\sigma=\tau^{\omega+q}$, then $\sigma^{(k)}=\left(\tau^{(k)}\right)^{k+q}$, and conditions imposed on $k$ guarantee  that $k+q\ge2$. Conversely, suppose that an ordinarization of the sword $\sigma$ is equal to $v^n$ for some word $v$ and some $n\ge2$. Then a $\mathbb{Z}$-unary word representing $\sigma$ should decompose as $\alpha_1\cdots\alpha_n$, where the ordinarization sends each factor $\alpha_i$, $i=1,\dots,n$, to the word $v$. If the parameters $\ell$ and $k$ are calculated from this representation of the sword $\sigma$ and $\tau_i$ is the normal sword represented by $\alpha_i$, $i=1,\dots,n$, we have $\sigma=\tau_1\cdots\tau_n$ and $\tau^{(k)}_i=v$ for each $i=1,\dots,n$. Since $|\wh{\tau_i}|\le |\wh{\sigma}|$, we can apply Lemma~\ref{Lemma_Equality_to_Burnside} to the swords $\tau_i$ and obtain that they are all equal. Therefore $\sigma$ is periodic.
\end{proof}

A central concept in Guba's papers~\cite{Guba1, Guba2} is the notion of the reduced form of a word relative to the Burnside variety $\mathfrak{B}_{m,k}$.  We are not going to reproduce here the original definition of this notion  because it is quite involved (it is given by simultaneous induction on five other notions) and because, as the reader will see, we can easily bypass the definition by working with a sufficient condition for a word to be in the reduced form (see Lemma~\ref{Lemma_Burnside_No_Long_words} below). First we reproduce a result which is crucial for our considerations; it constitutes the main point of the proof of Theorem 4.1 in~\cite{Guba1}.

\begin{Theorem}\label{THM_Guba_equal_reduced_forms}
Let $m\ge3$, $k\ge1$, and let $u$ and $v$ be words. The identity $u\is v$ holds in the variety $\mathfrak{B}_{m,k}$ if and only if $u$ and $v$ have the same reduced forms relative to $\mathfrak{B}_{m,k}$.
\end{Theorem}

Now we again assume that $k>m$ and, as above, let $\ell=k-m$. For short, words that are in the reduced form relative to $\mathfrak{B}_{m,k}$ will be called $\mathfrak{B}_{m,k}$-\emph{reduced} words. The following result that is a
consequence of Lemma~2.2 in~\cite{Guba1} is sufficient for our purposes.

\begin{Lemma}
\label{Lemma_Burnside_No_Long_words}
Let $k>m\ge3$ and $\ell=k-m$. A word is $\mathfrak{B}_{m,k}$-reduced if it contains no factors of the form $v^{2k-\ell-2}$.
\end{Lemma}

We use Lemma~\ref{Lemma_Burnside_No_Long_words} to prove that, for a normal sword $\sigma$ and large enough $k$, the word $\sigma^{(k)}$ is $\mathfrak{B}_{m,k}$-reduced. We start with considering swords of height 1.

\begin{Lemma}
\label{Lemma_Normal_to_reduced_Height_1}
Let $\sigma$ be a normal sword of height $1$ and $\ell>d(\sigma)$, $k-\ell-2>|\wh{\sigma}|$.
Then the word $\sigma^{(k)}$ is $\mathfrak{B}_{m,k}$-reduced.
\end{Lemma}

\begin{proof}
Since $h(\sigma)=1$, we have
\[
\sigma=u_0v_1^{\omega+q_1}u_1\cdots v_n^{\omega+q_n}u_n,
\]
for some ordinary words $u_0,u_1,\dots u_n,v_1,\dots,v_n$. Suppose that for some word $v$, the word $\sigma^{(k)}$ has a factor of the form $v^{2k-\ell-2}$. First assume that $v=v_i$ for some $i=1,\dots,n$. Since the sword $\sigma$ is normal, Lemma~\ref{Lemma_Periodic_Square} implies that the sword $v_{i-1}^2u_{i-1}v_i^{\omega+q_i}u_iv_{i+1}^2$ is not a factor of any $v_i$-periodic sword. Therefore, $v_i^{2k-\ell-2}$ is a factor of the word $v_{i-1}^2u_{i-1}v_i^{k+q_i}u_iv_{i+1}^2$. Consider the word $v_{i-1}^2u_{i-1}v_i^{q'_i}u_iv_{i+1}^2$ where $q'_i=2$ if $q_i\leq 2$ and $q'_i=q_i$ otherwise. Its length is not larger than $|\wh{\sigma}|$. But $|v_i^{k-\ell-2}|>k-\ell-2>|\wh{\sigma}|$. This implies that $v_i^{2k-\ell-2}$ cannot be a factor of $v_{i-1}^2u_{i-1}v_i^{k+q_i}u_iv_{i+1}^2$.

Thus, the word $v^{2k-\ell-2}$ is a product of words of the form $t_1v_i^{k+q_i}\ldots v_j^{k+q_j}t_2$ where $t_1$ is either a suffix of $u_i$ or equals to $v_{i-1}^{s_1}u_i$, $t_2$ is either a prefix of $u_{j+1}$ or equals to $u_{j+1}v_{j+1}^{s_2}$.  Let $L=\max\{|v_i|\mid \text{$v_i$ is a factor of the word $v$}\}$. Then
$|v|^{2k-\ell-2}\le |\wh{\sigma}|+Lk$. On the other hand, $|v|^{2k-\ell-2}\ge |v|k+k-\ell-2>|v|k+|\wh{\sigma}|$. This implies $|v|<L$, a contradiction.
\end{proof}

Now we consider the swords of larger height.

\begin{Lemma}~\label{Lemma_Normal_To_Reduced_Forms}
Let $\sigma$ be a normal  sword of height $h>1$ then, for all integers $k,p$ such that $p>p(\sigma)$ and $k-\ell-2>|\wh{\sigma}|$, the word $\sigma^{(k)}$ is a $\mathfrak{B}_{m,k}$-reduced form.
\end{Lemma}

\begin{proof}
We are going to prove this lemma by induction on height of the sword $\sigma$ using Lemma~\ref{Lemma_Normal_to_reduced_Height_1} as the basis of induction.

Let $\pi_0\rho_1^{\omega+q_1}\pi_1\ldots\rho_n^{\omega+q_n}\pi_n$ be a height representation of the sword $\sigma$.  Consider the sword $$\widetilde{\sigma}=\pi_0^{(k)}(\rho^{(k)}_1)^{\omega+q_1}\pi_1^{(k)}\ldots(\rho_n^{(k)})^{\omega+q_n}\pi_n^{(k)}$$ of height $1$.
Note that $\sigma^{(k)}=\widetilde{\sigma}^{(k)}$.

By the induction hypothesis  we have that all the words $\pi_i^{(k)}$ and $\rho_i^{(k)}$ are $\mathfrak{B}_{m,k}$-reduced forms. Let us prove now that $\widetilde{\sigma}$ is a normal sword. Indeed, the words $\rho_i$ are  not periodic by Lemma~\ref{Lemma_Periodic_To_Burnside}. Next, suppose that, for some index $i$, $(\rho_i^{(k)})^{\omega+q_i}\pi_i(\rho_{i+1}^{(k)})^{\omega+q_{i+1}}=(uv)^{\omega+q_i}u(vu)^{\omega+q_{i+1}}$. Let $\alpha_1, \alpha_2, \beta_1, \beta_2$ be swords such that $\alpha_1\alpha_2=\rho_i,\beta_1\beta_2=\rho_{i+1}$ and $\alpha_1^{(k)}=\beta_2^{(k)}=\pi_i^{(k)}=u, \alpha_2^{(k)}=\beta_1^{(k)}=v$. Again by Lemma~\ref{Lemma_Equality_to_Burnside} the following equalities hold $\alpha_1=\beta_2=\pi_i$, $\alpha_2=\beta_1$. Thus, $(\alpha_1\alpha_2)^{\omega+q_i}\alpha_1(\alpha_2\alpha_1)^{\omega+q_{i+1}}$ is a subsword of the normal sword $\sigma$. This implies that the sword $\widetilde{\sigma}$ is normal.

We proceed assuming that the word $\widetilde{\sigma}$ is not a $\mathfrak{B}_{m,k}$-reduced form and, for some word $T$, it contains $T^{2k-\ell-2}$ as a factor. Similarly to the proof of Lemma~\ref{Lemma_Normal_to_reduced_Height_1} suppose firstly that, for some index $i$, $T=\rho_i^{(k)}$. Again the word $\tau=(\rho_{i-1}^{(k)})^2\pi_{i-1}^{(k)}(\rho_{i-1}^{(k)})^{\omega+q_i}\pi_i^{(k)}(\rho_{i+1}^{(k)})^2$ can not be some factor of some $T$-periodic sword by Lemma~\ref{Lemma_Periodic_Square}. Thus, $|\tau|>|T^{2k-\ell-2}|$ and the length of $\tau^{(k)}$ does not exceed $|\wh{\sigma}|+|\rho_i^{(k)}|k$. On the other hand, $|T^{2k-\ell-2}|=|\rho_i^{(k)}|(2k-\ell-2)>|\rho_i^{(k)}|k+|\rho_i^{(k)}||\wh{\sigma}|$, a contradiction.

Suppose now  that $T$ has the form $t_1(\rho_i^{(k)})^{k+q_i}\ldots (\rho_j^{(k)})^{k+q_j}t_2$ where $t_1$ is a suffix of   $\rho^{(k)}_{i-1}\pi^{(k)}_{i-1}$, $t_2$ is a prefix of $\pi_{j+1}\rho^{(k)}_{j+1}$. Let $L$ be the maximum of $|\rho_i^{(k)}|$ where $\rho_i^{(k)}$ is a factor of $T$. Therefore, $$|T|\leq Lk+|t_1(\rho_i^{(k)})^{q'_i}\ldots (\rho_j^{(k)})^{q'_j}t_2|\leq Lk+L|\wh{\sigma}|.$$ But $|T|^{2k-\ell-2}=|T|k+|T|(k-\ell-2)>|T|(k+|\wh{\sigma}|)>L(k+|\wh{\sigma}|)$.

Hence, the word $\widetilde{\sigma}^{(k)}$ that equals to $\sigma^{(k)}$ is a $\mathfrak{B}_{m,k}$-reduced form.
\end{proof}

These lemmas lead us to the important

\begin{Proposition}~\label{Prop_Equality of omega through (k)-words}
  Let $\sigma_1,\sigma_2$ be two normal swords then the equality  $\sigma_1\is\sigma_2$ holds in $\mathfrak{E}$ ($\mathfrak{E}_{fin}$) if and only if $\sigma_1$ and $\sigma_2$ are equal.
\end{Proposition}
\begin{proof}
 Remind that every Burnside semigroup is an epigroup. Also, in view of Proposition~\ref{Prop_Finite_Members_Burnside} the free Burnside  semigroup satisfies any identity that holds in $\mathfrak{E}_{fin}$. Thus, the identities $\sigma_1\is\sigma_2$ and, therefore, $\sigma^{(k)}_1\is\sigma_2^{(k)}$ holds in the variety $var[x^{2k-\ell}\is x^{k-\ell}]$ for large enough $k,\ell$ and $m=k-\ell$. By Lemmas~\ref{Lemma_Normal_to_reduced_Height_1} and~\ref{Lemma_Normal_To_Reduced_Forms} the words $\sigma^{(k)}_1$ and $\sigma_2^{(k)}$ are
  $\mathfrak{B}_{m,k}$-reduced forms.
 Then by Theorem~\ref{THM_Guba_equal_reduced_forms} the words $\sigma^{(k)}_1$ and $\sigma^{(k)}_2$ are  equal.
 Hence, by Lemma~\ref{Lemma_Equality_to_Burnside} the swords $\sigma_1$ and $\sigma_2$ coincide.
   % and by Lemmas~\ref{Lemma_Normal_to_reduced_Height_1} and~\ref{Lemma_Normal_To_Reduced_Forms} the swords $\sigma_1$ and $\sigma_2$ coincide.

\end{proof}

\begin{proof}[ of Theorem~\ref{THM_Main}]
 Let $\sigma_1\is \sigma_2$ be any identity that holds in $\mathfrak{E}_{fin}$ or $\mathfrak{E}$. By Proposition~\ref{Prop_Normalization} there is an algorithm that constructs normal forms $\alpha_1,\alpha_2$ of the swords $\sigma_1,\sigma_2$ respectively and the identity $\alpha_1\is\alpha_2$, by Proposition~\ref{Prop_Nornal_Equalities},  also holds in $\mathfrak{E}_{fin}$ or $\mathfrak{E}$.
By Proposition~\ref{Prop_Equality of omega through (k)-words} this identity is equivalent to the equality of the swords $\alpha_1$ and $\alpha_2$ which we can check using algorithm from Proposition~\ref{Prop_LongestCommonPrefix}.

The following proof was given by M. Volkov in a verbal discussion. Let us prove that the identity basis of $\mathfrak{E}$ is infinite. For a prime number $p$, consider a non-trivial group $G$ satisfying the identity $x^p\is 1$. Let  $H$ be  the semigroup $G$ with adjoined $0$. Define a unary operation on $H$ by the following rule: for all $x\neq 1$ we set $\ov{x}=0$, and $\ov{1}=1$.
It is clear that the unary semigroup $H$ satisfies the first five identities.

Let $q$ be a prime number distinct from $p$ and let $x\in H$. Then $x^q\neq 1$ if and only if $x\neq 1$. Therefore, for $x\neq 1$, $\ov{x^q}=0=\ov{x}^q$. Thus, the identity $\ov{x^q}\is \ov{x}^q$ holds in $H$. Meanwhile, taking $x\in G$ distinct from $1$
we obtain $1=\ov{x^p}\neq \ov{x}^p=0$. Hence, the identity $\ov{x^p}\is\ov{x}^p$ fails in $H$.
\end{proof}

\section{Acknowledgements}
\label{sec:ack}

The author wants to thank M. Volkov who drew her attention to this problem and helped to improve the paper, the anonymous reviewer for many valuable remarks and suggestions, A. Popovich for help in finishing this work and opportunity to discuss the results and L. Shevrin for useful comments. The author acknowledges support from the Presidential Programme ``Leading Scientific Schools of the Russian Federation'', project no.\ 5161.2014.1, the Russian Foundation for Basic Research, project no.\ 14-01-00524, the Ministry of Education and Science of the Russian Federation, project no.\ 1.1999.2014/K, and the Competitiveness Program of Ural Federal University.

\bibliographystyle{abbrv}
\bibliography{epigroups}
\label{sec:biblio}

\end{document}